\documentclass{amsart}

\newcommand{\cA}{\mathcal{A}}

\newcommand{\cM}{\mathcal{M}}

\newcommand{\bF}{\mathbb{F}}

\newcommand{\bQ}{\mathbb{Q}}\newcommand{\bR}{\mathbb{R}}

\newcommand{\bZ}{\mathbb{Z}}

\usepackage[pagebackref,colorlinks,citecolor=Mahogany,linkcolor=Mahogany,urlcolor=Mahogany,filecolor=Mahogany]{hyperref}
\usepackage[dvipsnames,svgnames,x11names,hyperref]{xcolor}
\usepackage{url,graphicx,verbatim,amssymb,enumerate,stmaryrd,mathtools}
\usepackage{setspace}
\usepackage{microtype}
\usepackage[all,cmtip]{xy}
\usepackage[hmargin=2.75cm]{geometry}

\newtheorem{theorem}{Theorem}[section]
\newtheorem{lemma}[theorem]{Lemma}
\newtheorem{proposition}[theorem]{Proposition}
\newtheorem{corollary}[theorem]{Corollary}

\newtheorem{atheorem}{Theorem}

\theoremstyle{definition}
\newtheorem{definition}[theorem]{Definition}
\newtheorem{example}[theorem]{Example}
\newtheorem{remark}[theorem]{Remark}

\newtheorem{convention}[theorem]{Convention}

\newcommand{\mr}[1]{{\rm #1}}

\newcommand{\Diff}{\mr{DIFF}}
\newcommand{\Pl}{\mr{PL}}
\newcommand{\Top}{\mr{TOP}}
\newcommand{\IE}{\hyperlink{item.ie}{\textbf{IE}}}
\newcommand{\MT}{\hyperlink{item.mt}{\textbf{MT}}}
\newcommand{\WT}{\hyperlink{item.wt}{\textbf{WT}}}
\newcommand{\IT}{\hyperlink{item.it}{\textbf{IT}}}
\newcommand{\SH}{\hyperlink{item.sh}{\textbf{SH}}}
\newcommand{\KT}{\hyperlink{item.kt}{\textbf{KT}}}

\title{Proving homological stability for homeomorphisms of manifolds}
\author{Alexander Kupers}
\thanks{Alexander Kupers is supported by a William R. Hewlett Stanford Graduate Fellowship, Department of Mathematics, Stanford University, and was partially supported by NSF grant DMS-1105058.}

\date{\today}

\begin{document}

\begin{abstract}Following the argument for diffeomorphisms by Galatius and Randal-Williams, we prove that homeomorphisms of 1-connected manifolds of even dimension at least 6 exhibit homological stability. We deduce similar results for PL homeomorphisms and homeomorphisms as a discrete group.\end{abstract}

\maketitle

\tableofcontents

\section{Introduction} A major advance in the theory of manifolds is recent work by Galatius and Randal-Williams on diffeomorphism groups of higher-dimensional smooth manifolds \cite{grwcob,grwstab1}. They prove these groups exhibit homological stability, and compute the stable homology. In this note we explain that this homological stability argument also applies to homeomorphism groups of topological manifolds, PL homeomorphism groups of PL manifolds, and discrete groups of homeomorphisms. 

These results are not unexpected, so we feel compelled to justify this paper. Firstly, we hope that Section \ref{sec.input} gives a good overview of the relevant properties of topological manifolds and literature on them. We believe the tools described in that section are those needed to extend other homological stability results for diffeomorphisms to homeomorphisms and PL homeomorphisms. Secondly, these results point to the interesting problem of scanning for topological and PL manifolds discussed later in this introduction. Finally, the results of Galatius and Randal-Williams have already had many applications in the study of smooth manifolds and diffeomorphisms, e.g. \cite{oscarjohannesblock,oscarjohannesborisscalar,weissdalian}, and we hope that this result lays the groundwork for similar applications to topological manifolds and homeomorphisms.

\subsection{The results} Let $\Top_\partial$ be the topological group of homeomorphisms fixing the boundary pointwise (in the compact-open topology) and $M$ a $2n$-dimensional topological manifold with boundary $S^{2n-1}$. In Definition \ref{def.stab} we construct a \emph{stabilization map}
\[t: \Top_\partial(M) \to \Top_\partial(M\# (S^n \times S^n))\] 

Our main result is the homology of the group of homeomorphisms stabilizes as one takes connected sum with more and more copies of $S^n \times S^n$:

\begin{atheorem}\label{thm.main} Let $M$ be a compact 1-connected topological manifold of dimension $2n \geq 6$ with boundary $S^{2n-1}$. Then the map \[t_*: H_*\left(B\Top_\partial(M \#_g (S^n \times S^n)) \right) \to H_*\left(B\Top_\partial(M \#_{g+1} (S^n \times S^n))\right)\]
induced by the stabilization map is an isomorphism in the range $* \leq \frac{g-3}{2}$ and a surjection in the range $* \leq \frac{g-1}{2}$.\end{atheorem}

\begin{remark}The case $2n=2$ follows from Harer's theorem \cite{harerstab,wahlmcg}, the existence and uniqueness of smooth structures on topological surfaces \cite{rado,moisebook,hatchertorus} and Smale's theorem \cite{smaledisk}. These results together imply the inclusion $\Diff_\partial(\Sigma_{g,1}) \hookrightarrow \Top_\partial(\Sigma_{g,1})$ is a weak equivalence for any genus $g$ surface $\Sigma_{g,1}$ with one boundary component. We will thus only focus on the high-dimensional case $2n \geq 6$.\end{remark} 

This result admits a strengthening, Theorem \ref{thm.ext} in Section \ref{sec.ext}. This implies the result is also true for PL homeomorphisms, see Corollary \ref{thm.pl}. Using a theorem of McDuff we conclude that the result is also true for homeomorphisms as a discrete group, see Corollary \ref{thm.disc}.


\subsection{The stable homology} Scanning has been used to successfully compute the stable homology of mapping class groups \cite{madsenweiss} and diffeomorphisms of higher dimensional manifolds \cite{gmtw,grwcob}. We conjecture that the analogous results hold for topological manifolds, though we do not necessarily think they can be proven using the same techniques. 

If so, the following is true. Let $W_{g,1}$ denote $D^{2n} \# g(S^n \times S^n)$, then the stable homology of $B\Top_\partial (W_{g,1})$ is given by the homology of a component of the infinite loop space associated to the Thom spectrum $MT(\theta^\Top)$. Here $MT(\theta^\Top)$ is the Thom spectrum of $-(\theta^\Top)^*\gamma$, where $\Top(2n)$ is the topological group of homeomorphisms of $\bR^{2n}$, $\gamma$ is the universal $\bR^{2n}$-bundle over $B\Top(2n)$ and $\theta^\Top: B\Top(2n)\langle n \rangle \to B\Top(2n)$ is the $n$-connective cover.

In \cite{mauriciothesis}, Gomez Lopez used scanning to compute the homotopy type of the classifying space of the PL bordism category in terms of a spectrum of PL manifolds, which is conjecturally weakly equivalent to $MT\Pl(2n)$. We believe his techniques can be used to compute the stable homology of PL homeomorphisms of connected sums of $S^n \times S^n$. This amounts to extending his result to tangential structures and showing that the following property of the smooth bordism category also holds in the PL setting: the classifying space of the bordism category with tangential structure $\theta^\Pl: B\Pl(2n)\langle n \rangle \to B\Pl(2n)$ is weakly equivalent to that of a subcategory of highly-connected manifolds and bordisms. If this program can be completed, it implies the identification conjectured above when $n \geq 4$. To see this, note Theorem \ref{thm.plhomotopy} implies $B\Pl(2n)\langle n \rangle \to B\Top(2n)\langle n \rangle$ is a weak equivalence for $n \geq 3$ and by Remark \ref{rem.comppltop} $B\Pl_\partial (W_{g,1}) \to B\Top_\partial (W_{g,1})$ is a weak equivalence when $n \geq 4$.

What little is known about the (co)homology of $B\Top_\partial(W_{g,1})$ is consistent with these conjectures. In \cite{oscarjohannesblock}, Ebert and Randal-Williams define MMM-classes for topological manifold bundles, indexed by the cohomology of $H^*(BS\Top(2n);\bF)$ with $\bF$ a field, where $S\Top(2n) \subset \Top(2n)$ is the subgroup of orientation-preserving homeomorphisms.  Using this they prove that the map induced on $H^*(-;\bQ)$ by 
\[B\Diff_\partial(W_{g,1}) \to B\Top_\partial(W_{g,1})\]
is surjective for $* \leq \frac{g-4}{2}$ and split injective for $* \leq  \min(\frac{2n-7}{2},\frac{2n-4}{3})$. This is consistent because the stable cohomology of $B\Diff_\partial(W_{g,1})$ is that of a component of the infinite loop space of the Thom spectrum of $-(\theta^O)^*\gamma$ for $\theta^O: BO(2n)\langle n \rangle \to BO(2n)$ the $n$-connective cover, and the map induced by $MT(\theta^O) \to MT(\theta^\Top)$ on $H^*(-;\bQ)$ is injective in a range and always surjective. See \cite{weissdalian} for recent results about the rational cohomology of $B\Top(2n)$.

\subsection{Outline of the paper} In Section \ref{sec.proof} we will prove Theorem \ref{thm.main}. We will also try to follow the notation of \cite{grwstab1}, so that a reader familiar with their argument will find it easier to follow ours. Our proof requires several foundational results on topological manifolds, which are collected in Section \ref{sec.input}. Finally, in Section \ref{sec.ext} we give extensions and applications to PL homeomorphisms and discrete homeomorphisms.

\subsection{Acknowledgements} We thank S\o ren Galatius for many interesting conversations and much helpful advice. We thank Oscar Randal-Williams for helpful comments and a number of simplifications of the proofs of various lemma's. We also thank Rob Kirby, Sam Nariman and Frank Quinn for answering several questions.

\section{The Galatius-Randal-Williams argument for topological manifolds} \label{sec.proof} In this section we explain how to adapt to topological manifolds the argument of Galatius and Randal-Williams \cite{grwstab1} (see \cite{sorenoscarstability} for a less general, but easier version). The only results about topological manifolds needed for arguments in this paper are the following. We state them in abbreviated form here, leaving the precise technical statements to Section \ref{sec.input} (see Remark \ref{rem.plrefs} for a similar list in the PL category). 

\begin{description}
	\item[\hypertarget{item.ie}{IE}] \emph{Parametrized isotopy extension}, Theorem \ref{thm.ie}, says that an isotopy of embeddings of a submanifold can be extend to an ambient isotopy.
	\item[\hypertarget{item.it}{IT}] \emph{Immersion theory}, Theorem \ref{thm.it}, is an $h$-principle classifying immersions in terms of homotopy-theoretic data. 
	\item[\hypertarget{item.mt}{MT}] \emph{Microbundle transversality}, Theorem \ref{thm.mt}, says that any two submanifolds can be made transverse as long as one of them has a normal bundle.
	\item[\hypertarget{item.kt}{KT}] \emph{Kister's theorem}, Theorem \ref{thm.kt}, says that the inclusion $\Top(n) \to \mr{Emb}(\bR^n,\bR^n)$ is a weak equivalence.
	\item[\hypertarget{item.sh}{SH}] The \emph{stable homeomorphism theorem}, Theorem \ref{thm.sh}, says that any homeomorphism $\bR^n \to \bR^n$  is a composition of finitely many homeomorphisms which are the identity on some open subset.
	\item[\hypertarget{item.wt}{WT}] The \emph{Whitney trick}, Theorems \ref{thm.wt6}, for $2n \geq 6$ allows one to geometrically cancel algebraically canceling intersection  points.
\end{description}

The three subsections of this section contain the following steps of the proof:
\begin{enumerate}
\item The first subsection defines geometric models for the moduli space of topological manifolds and the stabilization map.
\item The second subsection defines a semisimplicial resolution of embedded cores, establishes all except one of its properties and finishes the homological stability argument.
\item The third subsection proves the one crucial property left unproven in the second subsection: that the semisimplicial simplicial set of embedded cores is highly-connected.
\end{enumerate}

Due to the technical details of topological manifolds, we need to use simplicial sets instead of spaces. A good reference for simplicial sets is \cite{goerssjardine}.

\subsection{Moduli spaces of topological manifolds}\label{subsec.modulispace}

Let $W$ be a 1-connected compact $2n$-dimensional topological manifold with boundary $S^{2n-1}$. Our goal is to give a geometric model for the classifying space of the homeomorphisms of $W$ that fix the boundary pointwise. In the smooth setting such a model is given by smooth submanifolds of $\bR^\infty$ diffeomorphic to $W$ satisfying a certain boundary condition. A similar model exists in the topological setting. It will be useful to use locally flat submanifolds, see Subsection \ref{subsec.basics} for definitions.

To make sure that the stabilization map preserves the locally flatness condition, we endow our submanifolds with germs of collars. Since $W$ has boundary $S^{2n-1}$, we can glue on an open collar $S^{2n-1} \times (-\epsilon,0]$ to obtain the manifold $W \cup (S^{2n-1} \times (-\epsilon,0])$. Let $\mr{std}$ denoting the standard embedding $S^{2n-1} \times (-\epsilon,0] \hookrightarrow \bR^{2n} \times (-\infty,0] \hookrightarrow \bR^{N-1} \times (-\infty,0]$.

\begin{definition}\label{def.embwg1} Let $N \geq 2n$. \begin{itemize}
\item  For $\epsilon > 0$, the simplicial set $\mr{Emb}^\mr{lf}_\epsilon(W,\bR^N)$ has $k$-simplices the locally flat embeddings $e$ fitting into a commutative diagram
\[\xymatrix{\Delta^k \times (W \cup (S^{2n-1} \times (-\epsilon,0])) \ar@{^(->}[rr]^-e \ar[rd] & & \Delta^k \times \bR^{N-1} \times \bR \ar[ld] \\
& \Delta^k & }\]
with the following two properties:
\begin{enumerate}[(i)]
	\item On $\Delta^k \times (S^{2n-1} \times (-\epsilon,0])$ the embedding $e$ equals $\mr{id} \times \mr{std}$.
	\item the preimage $e^{-1}(\Delta^k \times \bR^\infty \times (-\infty,0])$ consists of $\Delta^k \times S^{2n-1} \times (-\epsilon,0]$.
\end{enumerate}
\item Using the restrictions $\mr{Emb}^\mr{lf}_\epsilon(W,\bR^N) \to \mr{Emb}^\mr{lf}_{\epsilon'}(W,\bR^N)$ for $\epsilon'<\epsilon$ we define
\[\mr{Emb}^\mr{lf}_\partial(W,\bR^N) \coloneqq \underset{\epsilon \to 0}{\mr{colim}}\,\mr{Emb}^\mr{lf}_\epsilon(W,\bR^N)\]
\item Using the inclusions $\bR^{N} \hookrightarrow \bR^{N+1}$ we define \[\mr{Emb}^\mr{lf}_\partial(W,\bR^\infty) \coloneqq \underset{N \to \infty}{\mr{colim}}\,\mr{Emb}^\mr{lf}_\partial(W,\bR^N)\]
\end{itemize}
\end{definition}

In the smooth case, the space of embeddings of a manifold into $\bR^\infty$ is weakly contractible. The same is true for topological manifolds.

\begin{lemma}\label{lem.embcontract}We have that $\mr{Emb}^\mr{lf}_\partial(W,\bR^\infty)$ is weakly contractible.
\end{lemma}

\begin{proof}We first check it is non-empty. Corollary \ref{cor.whitney} proves the locally flat weak Whitney embedding theorem and the argument can easily be modified to give a locally flat embedding 
\[e_0: W \cup (S^{2n-1} \times (-\epsilon,0]) \hookrightarrow \bR^{N-1} \times \bR\]
that satisfies properties (i) and (ii) of Definition \ref{def.embwg1}. If we pick a collar $S^{2n-1} \times [0,2] \hookrightarrow W$, we can assume that $e_0$ is the standard inclusion on $S^{2n-1} \times [0,1]$ and this is the only intersection of $\mr{im}(e_0)$ with $\bR^{N-1} \times [0,1]$. 

We have that $\mr{Emb}^\mr{lf}_\partial(W,\bR^\infty)$ is Kan: since being Kan is preserved by colimits, it suffices to show that $\mr{Emb}^\mr{lf}_\epsilon(W,\bR^N)$ is Kan and this follows from Lemma \ref{lem.embkan}. Thus for weak contractibility we need to extend a map $\partial \Delta^i \to \mr{Emb}^\mr{lf}_\partial(W,\bR^\infty)$ to $\Delta^i$. That is, by Lemma \ref{lem.lfglue} (which depends on \IE) we have a locally flat embedding 
\[\xymatrix{\partial \Delta^i \times (W \cup (S^{2n-1} \times (-\epsilon',0]))  \ar@{^(->}[rr]^-e \ar[rd] & &  \partial \Delta^i \times \bR^{N'-1} \times \bR \ar[ld] \\
& \partial \Delta^i & }\]
and we need to extend it to $\Delta^i$, where we are allowed to decrease $\epsilon'$ and increase $N'$. We will do this by providing a three-step homotopy to a constant family.

\begin{enumerate}[(i)]
	\item Earlier we picked a collar $S^{2n-1} \times [0,2] \hookrightarrow W$. We start by showing that we can make $e$ the standard inclusion on $S^{2n-1} \times [0,1]$ and have this be the only intersection of $\mr{im}(e)$ with $\bR^{N'-1} \times [0,1]$. Let $\mr{std}$ denote the standard inclusion of $S^{2n-1} \times [0,1]$. For $t \in [0,1]$ and $\theta \in \partial \Delta^{i}$
	\[e^{(1)}_t(\theta)(w) = \begin{cases} (\mr{std}(w)) & \text{if $w \in S^{2n-1} \times (-\epsilon,t]$} \\
		(e(\theta)(\phi,2 \cdot \frac{r-t}{2-t})+t \cdot e_1) & \text{if $w = (\phi,r) \in S^{2n-1} \times [t,2]$} \\
		(e(\theta)(w)+t \cdot e_1) & \text{otherwise}\end{cases}\]
	\item We insert $e_0$ into $N'$ additional coordinates: for $t \in [0,1]$ and $\theta \in \partial \Delta^i$
	\[e^{(2)}_t(\theta)(w) = \begin{cases} (\mr{std}(w),0) & \text{if $w \in S^{2n-1} \times (-\epsilon,0]$} \\
	(\mr{std}(w),tr\cdot\mr{std}(w)) & \text{if $w = (\phi,r) \in S^{2n-1} \times [0,1]$} \\
	(e^{(1)}_1(\theta)(w),t\cdot e_0(w)) & \text{otherwise}\end{cases}\]
	where the notation $(x,y)$ means that $x \in \bR^{N'-1} \times \bR$ and $y \in \bR^{N'-1} \times \bR$.
	\item Next we dampen $e$ away: for $t \in [0,1]$ and $\theta \in \partial \Delta^{i}$
		\[e^{(3)}_t(\theta)(w) = \begin{cases} (\mr{std}(w),0) & \text{if $w \in S^{2n-1} \times (-\epsilon,0]$} \\
		((1-tr)\cdot \mr{std}(w),r \cdot\mr{std}(w)) & \text{if $w = (\phi,r) \in S^{2n-1} \times [0,1]$} \\
		((1-t)\cdot e^{(1)}_1(\theta)(w),e_0(w)) & \text{otherwise}\end{cases}\]
\end{enumerate}

The end result is a constant family given by the locally flat embedding
		\[e^{(3)}_1(\theta)(w) = \begin{cases} (\mr{std}(w),0) & \text{if $w \in S^{2n-1} \times (-\epsilon',0]$} \\
		((1-r)\cdot\mr{std}(w),r\cdot\mr{std}(w)) & \text{if $w = (\phi,r) \in S^{2n-1} \times [0,1]$} \\
		(0,e_0(w)) & \text{otherwise}\end{cases}\]

To check that the this family of embeddings is locally flat, we note that Lemma \ref{lem.lfglue} tells us it suffices to check this for each of the three steps individually. For step (i) this is clear and for steps (ii) and (iii) one uses Lemma \ref{lem.lfprod}. \end{proof}

In the smooth setting, the previous result is usually deduced from a quantitative version of the Whitney embedding theorem. This says that the space of embeddings into $\bR^N$ is highly-connected for large $N$. The following remarks address the question of a quantitative Whitney embedding theorem for topological embeddings.

\begin{remark}We will prove that if $N \geq 4n+2$, then $\mr{Emb}^\mr{lf}_\partial(W,\bR^N)$ is path-connected. For convenience we will assume $n \geq 3$. We are given a map $\partial \Delta^1 \to \mr{Emb}^\mr{lf}_\partial(W,\bR^N)$, which we want to extend to an embedding over $\Delta^1$. Extend it by any continuous map, and apply the relative version of general position as in Theorem \ref{thm.gp}. When $N \geq 4n+2$, the result is an extension to a locally flat embedding $e: \Delta^1 \times W \to \Delta^1 \times \bR^N$. This is not a family of embeddings over $\Delta^1$, but a concordance. Now use that concordance of embeddings implies isotopy in the range $N \geq 2n+3$ \cite{pedersenconcordances}.\end{remark}

\begin{remark}If $W$ admits a smooth structure, one can further prove that $\mr{Emb}^\mr{lf}_\partial(W,\bR^N)$ is $i$-connected when $N \geq 4n+2i+4$ using smoothing theory for embeddings. By the previous remark we can ignore $\pi_0$. Theorem A of \cite{lashofembeddings} says that 
\[\mr{hofib}\left[\mr{Emb}^\Diff_\partial(W,\bR^N) \to \mr{Emb}^\mr{lf}_\partial(W,\bR^N)\right] \to \mr{hofib}\left[\mr{Imm}^\Diff_\partial(W,\bR^N) \to \mr{Imm}^\mr{lf}_\partial(W,\bR^N)\right]\]
is an isomorphism on $\pi_i$ for $i>0$. In the smooth setting, the quantitative Whitney embedding theorem and immersion theory imply that $i$th homotopy groups of $\mr{Emb}^\Diff_\partial(W,\bR^N)$ and $\mr{Imm}^\Diff_\partial(W,\bR^N)$ vanish for $4n+2i+3 \leq N$ and $4n+i+1 \leq N$ respectively. This implies that 
\[\pi_i(\mr{Emb}^\mr{lf}_\partial(W,\bR^N))) = \pi_i(\mr{Imm}^\mr{lf}_\partial(W,\bR^N))\]
when $0<2i \leq N-4n-4$. Immersion theory (\IT) identifies the right hand space with a space of sections of a bundle with fiber given by the topological Stiefel manifold $\Top(N)/\Top(N,2n)$. This is highly-connected:  $\pi_i(\Top(N)/\Top(N,2n))=0$ for $2n+i+1 \leq N$ by the first remark on page 147 of \cite{lashofembeddings}. By obstruction theory the homotopy groups of the section space vanish for $4n+i+1 \leq N$, so we get the desired range.
\end{remark}

We will use Lemma \ref{lem.embcontract} to give a geometric model for the classifying space of the group of homeomorphisms. Instead of thinking of homeomorphisms as a topological group in the compact-open topology, it is more convenient to take the singular simplicial set and consider them as a simplicial group:

\begin{definition}Let $\Top_\partial(W)$ be the simplicial group with $k$-simplices given by the homeomorphisms $\varphi$ that fit in a commutative diagram
\[\xymatrix{\Delta^k \times W\ar[rr]_\cong^\varphi \ar[rd] & & \Delta^k \times W \ar[ld] \\
 & \Delta^k & }\]
and are the identity on $\Delta^k \times \partial W$.
 \end{definition}

The following lemma will be used to check that the quotient of locally flat embeddings by homeomorphisms is a model for $B\Top_\partial(W)$. Here $\Top_\partial(W)$ acts on $\mr{Emb}^\mr{lf}_\partial(W,\bR^\infty)$ by extending the homeomorphisms to be the identity on $S^{2n-1} \times (-\epsilon,0]$.

\begin{lemma}\label{lem.kanfibration} Let $G$ be a simplicial group and $X$ a non-empty Kan simplicial set with levelwise free $G$-action, then we have that
\begin{enumerate}[(i)]
\item $X \to X/G$ is a Kan fibration,
\item if $Y \to Z$ is a Kan fibration, then $X \times_G Y \to X \times_G Z$ is a Kan fibration, and
\item if $Y \to Z$ is a weak equivalence, then $X \times_G Y \to X \times_G Z$ is a weak equivalence.
\end{enumerate} \end{lemma}

\begin{proof}Part (i) is Corollary V.2.6 of \cite{goerssjardine}. For part (ii), consider the diagram 
\[\xymatrix{& X \times Y \ar@{->>}[d] \ar@{->>}[r] & X \times Z \ar@{->>}[d] \\
& X \times_G Y \ar[r] & X \times_G Z  \\
& \Lambda^n_i \ar[r] \ar[u]^f \ar@{.>}@/^8ex/[uu]^{\tilde{f}} & \Delta^n \ar@{.>}[luu] \ar[u]_F}\]

We need to lift $F$ to $X \times_G Y$. We start by remarking the top map is a fibration since pullbacks preserve fibrations and the top vertical maps are Kan fibrations by part (i).

We next lift $f: \Lambda^n_i \to X \times_G Y$ to $X \times Y$. To do so, we note that we can lift a single vertex, since $X \times Y \to X \times_G Y$ is surjective. Now use that $\{v\} \hookrightarrow \Lambda^n_i$ is an acyclic cofibration and $X \times Y \to X \times_G Y$ is a Kan fibration, so we can lift the remainder of the horn to a map $\tilde{f}: \Lambda^n_i \to X \times Y$. Then the fact that $X \times Y \to  X \times Z  \to X \times_G Y$ is a Kan fibration gives us a lift $\tilde{F}: \Delta^n \to X \times Y$. Composing $\tilde{F}$ with the map $X \times Y \to X \times_G Y$ gives the desired lift of $F$ to $X \times_G Y$.

For part (iii), note there is a commutative diagram 
\[\xymatrix{X \times Y \ar[r]^\simeq \ar@{->>}[d] & X \times Z \ar@{->>}[d] \\
X \times_G Y \ar[r] & X \times_G Z}\]
with top map a weak equivalence. Part (i) implies the vertical maps are surjective Kan fibrations. Since the map on fibers is the identity map $G \to G$, the long exact sequence of homotopy groups implies the bottom map is a weak equivalence as well. 
\end{proof}

\begin{lemma}We have that $\mr{Emb}^\mr{lf}_\partial(W,\bR^\infty)/ \Top_\partial(W)$ is a classifying space for $\Top_\partial(W)$.
\end{lemma}

\begin{proof}By Lemma \ref{lem.embcontract}, $\mr{Emb}^\mr{lf}_\partial(W,\bR^\infty)$ is weakly contractible. The result then follows from part (i) of Lemma \ref{lem.kanfibration} once we remark that the action of the $k$-simplices of $\Top_\partial(W)$ on the $k$-simplices of $\mr{Emb}^\mr{lf}_\partial(W,\bR^\infty)$ is free.\end{proof}

\begin{definition}We define
\[\cM(W) \coloneqq \mr{Emb}^\mr{lf}_\partial(W,\bR^\infty)/ \Top_\partial(W)\]
and call this the \emph{moduli space of topological manifolds homeomorphic to $W$}.\end{definition}

We now give a geometric model for the stabilization map. Let $W^\epsilon_{1,2}$ be $(S^{2n-1} \times (-\epsilon,1]) \# (S^n \times S^n)$, with connected sum along a locally flat embedding $D^{2n} \hookrightarrow S^{2n-1} \times (1/4,3/4)$. Using a construction as in Corollary \ref{cor.whitney}, for $N$ sufficiently large there exists a locally flat embedding of $W^\epsilon_{1,2}$ into $\bR^N \times (-\epsilon,1]$ which coincides near $\bR^N \times (-\epsilon,0]$ and $\bR^N \times \{1\}$ with the standard embedding $S^{2n-1} \times (-\epsilon,1] \hookrightarrow \bR^N \times (-\epsilon,1]$. See Figure \ref{fig.wpesilon}.

\begin{definition}\label{def.stab}The \emph{stabilization map} $t$ is given by
\begin{align*} t: \cM(W) &\to \cM(W \# (S^n \times S^n)) \\
  X &\mapsto (X+e_1) \cup W^\epsilon_{1,2}\end{align*}
\end{definition}

\begin{figure}
\centering{
\resizebox{75mm}{!}{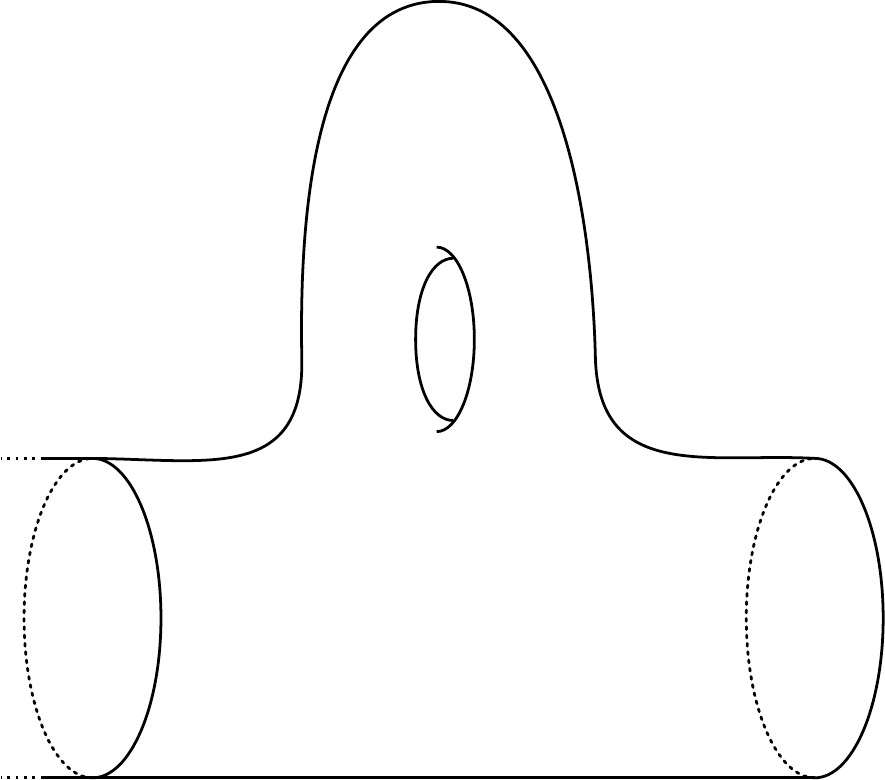}
\caption{The manifold $W_{1,2}^\epsilon$.}
\label{fig.wpesilon}
}
\end{figure}

\subsection{Resolving by cores} To prove homological stability, we build a semisimplicial resolution of $\cM(W)$ by embedded cores. The idea is as follows: it would be easy to prove homological stability if the stabilization map $t$ had an inverse. Such an inverse would exist if there was canonical choice of embedded copy of $W_{1,1}$. There is of course no such choice. However, as long as a certain complex of choices is highly-connected, we can still obtain a spectral sequence that allows us to prove homological stability.

\subsubsection{The resolution} We start by defining the objects which use resolve $\cM(W)$, see Figure \ref{fig.core} for an example.

\begin{definition}\label{def.cores} \begin{itemize}
\item Let $D^{2n} \hookrightarrow S^n \times S^n$ be a locally flat embedding that avoids $S^n \vee S^n$. The \emph{standard thick core} $H$ is given by 
\[H \coloneqq ((S^n \times S^n) \backslash \mr{int}(D^{2n})) \cup ([0,1] \times D^{2n-1})\] where the glueing happens along the lower hemisphere $\{1\} \times D^{2n-1} \hookrightarrow \partial ((S^n \times S^n) \backslash \mr{int}(D^{2n})) \cong S^{2n-1}$.
\item The \emph{standard core} $C \subset H$ is given by \[C \coloneqq (S^2 \vee S^2) \cup \gamma\subset W\] where $\gamma$ is a locally flat embedded arc in $W$ from a point in $S^2 \vee S^2$ to $(0,0) \in [0,1] \times D^{2n-1}$, which avoids $S^2 \vee S^2$ and coincides with $[0,1] \times \{0\}$ in $[0,1] \times D^{2n-1}$.
\end{itemize}\end{definition}

\begin{figure}
\centering{
\resizebox{75mm}{!}{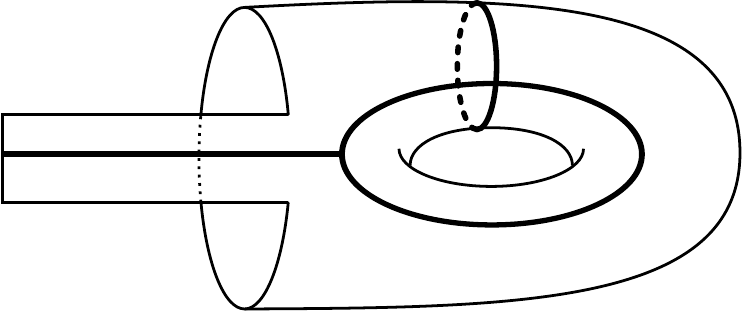}
\caption{The thick core $H$ containing a core $C$ for $n=1$.}
\label{fig.core}
}
\end{figure}

\begin{remark}A way to obtain $(S^n \times S^n) \backslash \mr{int}(D^{2n})$ is by \emph{plumbing} two copies of $S^n \times D^n$. That is, take two copies of $S^n \times D^n$ and let $D^n \subset S^n $ be the top hemisphere. Then identify the two copies of $D^n \times D^n \subset S^n \times D^n$ using the map $D^n \times D^n \to D^n \times D^n$ given by $(x,y) \mapsto (-y,x)$. We will eventually construct embedded thick cores $H$ by attaching a strip to such plumbings.\end{remark}

The strip $[0,1] \times D^{2n-1}$ is used to attach $H$ to the boundary. This should be done in a standard way, the definition of which requires us to fix a \emph{(germ of a) collar chart}. This is an equivalence class of a pair of $\delta>0$ and an embedding  
\[\eta: (-\delta,\delta) \times \bR^{2n-1} \to W \cup (S^{2n-1} \times (-\delta,0])\]
which on $(-\delta,0] \times \bR^{2n-1}$ coincides with $\mr{id} \times \hat{\eta}'$ for an embedding $\hat{\eta}': \bR^{2n-1} \to S^{2n-1}$. The equivalence relation is generated by $(\delta,\eta) \sim (\delta',\eta')$ if $\delta' < \delta$ and $\eta|_{(-\delta,\delta) \times \bR^{2n-1}} = \eta'$.

Recall that a semisimplicial simplicial set $X_\bullet$ is a sequence of simplicial sets $X_q$ for $q \geq 0$ with face maps $d_i: X_q \to X_{q-1}$ for $q \geq 0$ and $0 \leq i \leq q$, satisfy the standard relations between face maps.

\begin{definition}\label{def.k} Fix a germ  of  collar chart $[\delta,\eta]$. We define a semisimplicial simplicial set $K_\bullet(W)$ (where $\bullet$ denotes the semisimplicial direction). 
\begin{itemize}
\item The $0$-simplices $K_0(W)$ in semisimplicial direction, have $k$-simplices given by the colimit as $\epsilon \to 0$ of pairs $(t,\varphi)$ of a map $t: \Delta^k \to \bR$ and a $k$-simplex $\varphi$ of the simplicial set 
\[\mr{Emb}^\mr{lf}\left(H \cup ((-\epsilon,0] \times D^{2n-1}),W  \cup ((-\epsilon,0] \times S^{2n-1})\right)\]
such that $\varphi(s,r) = (s,r+te_2)$ for $(s,r)$ in the coordinates $(-\delta,\delta) \times D^{2n-1}$ of a representative of $[\delta,\eta]$.
\item The $q$-simplices $K_q(W)$ in semisimplicial direction, have $k$-simplices given by $(q+1)$-tuples $((t_0,\varphi_0),\ldots,(t_q,\varphi_q))$ of $k$-simplices of $K_0(W)$ such that $t_0<\ldots<t_q$ and for all $t \in \Delta^k$ the images of the $\varphi_i$ are disjoint.
\end{itemize}
\end{definition}

An argument similar to that in Lemma \ref{lem.embkan} proves:

\begin{lemma}\label{lem.kqkan}For each $q$, the simplicial set $K_q(W)$ is Kan.\end{lemma}


A resolution is a particular type of augmented semisimplicial simplicial set $X_\bullet$. This is a sequence of simplicial sets $X_q$ for $q \geq -1$ with face maps $d_i: X_q \to X_{q-1}$ for $q \geq 0$ and $0 \leq i \leq q$ and augmentation map $\epsilon: X_0 \to X_{-1}$. The $d_i$ satisfy the standard relations between face maps and the augmentation $\epsilon$ satisfies $\epsilon d_0 = \epsilon d_1$.

\begin{definition}We define an augmented semisimplicial simplicial set $X_{\bullet}(W)$ (with $\bullet$ denoting the semisimplicial direction). For $q \geq 0$, the $q$-simplices in the semisimplicial direction are given by the simplicial set 
\[X_{q}(W) \coloneqq \mr{Emb}_\partial(W,\bR^\infty) \times_{\Top_\partial(W)} K_q(W)\]
We set $X_{-1}(W) = \cM(W)$ and the unique map $K_0(W) \to *$ induces an augmentation 
\[\epsilon: X_{0}(W) \to \cM(W)\]
\end{definition}

The following proposition is the core of the argument and will be proven in the next subsection as Proposition \ref{prop.kconn}. We will assume it for the moment. 

\begin{definition}\label{def.gw} We define the number  $g(W)$ by
\[g(W):= \max\{g \mid \text{we can embed $g$ disjoint copies of $(S^n \times S^n) \backslash \mr{int}(D^{2n})$ into $W$}\}\]
\end{definition}

\begin{proposition}\label{prop.connk}The semisimplicial simplicial set $K_\bullet(W)$ is $ \frac{g(W)-4}{2} $-connected.\end{proposition}

The homological stability argument will use the geometric realization spectral sequence. For an augmented semisimplicial simplicial set $X_\bullet$ this spectral sequence is given by
\[(E^1_{p,q},d^1) = \left(H_q(X_p),{\textstyle \sum}_i (-1)^i (d_i)_*\right) \Rightarrow H_{p+q+1}(X_{-1},||X_\bullet||)\]
where $||X_\bullet||$ is the simplicial space obtained as the (fat) geometric realization of the semisimplicial direction of $X_\bullet$. That is, the $E^1$-page has non-zero columns for $p \geq -1$, given by the homology of the $p$-simplices. The $d^1$-differential $d^1: E^1_{p+1,q} \to E^1_{p,q}$ is  the alternating sum of the maps induced by the face maps when $p \geq 0$, and is the map induced by augmentation when $p = -1$. 

\subsubsection{Identifying the spectral sequence} For $X_\bullet(W)$, we do the following:
\begin{itemize}
\item in Proposition \ref{prop.resconn} we identify the target of the spectral sequence,
\item in Proposition \ref{prop.rese1} we identify the entries on the $E^1$-page, and
\item in Proposition \ref{prop.resd1} we identify the $d^1$-differential.
\end{itemize} 

We start with the target. This uses the following result of Rezk, Proposition 5.4 of \cite{rezkfib}, which describes a situation when taking homotopy fibers commutes with geometric realization. It replaces a similar result in the topological setting, Lemma 2.1 of \cite{oscarresolutions}. These results compares two constructions of a homotopy fiber. If $X_\bullet \to Y_\bullet$ is a map of simplicial spaces and we pick a $0$-simplex $y_0: \Delta^0 \to Y_\bullet$, we can form the \emph{homotopy fiber after realization} $\mr{hofib}_{|y_0|}(|X_\bullet| \to |Y_\bullet|)$. However, we can also take the \emph{homotopy fiber before realization} $\mr{hofib}_{y_0}(X_\bullet \to Y_\bullet)$, given by $[q] \mapsto \mr{hofib}_{(y_0)_q}(X_q \to Y_q)$.

\begin{proposition}[Rezk] \label{prop.rezk} Let $X_\bullet \to Y_\bullet$ be a map of Reedy cofibrant simplicial spaces such that $[q] \mapsto \pi_0(Y_q)$ is a discrete simplicial set. Then we have that
\[\mr{hofib}_{y_0}(|X_\bullet| \to |Y_\bullet|) \simeq |\mr{hofib}_{y_0}(X_\bullet \to Y_\bullet)|\]\end{proposition}

\begin{proposition}\label{prop.resconn} The map $||\epsilon||: ||X_\bullet(W)|| \to \cM(W)$ obtained as the realization of the augmentation map is $ \frac{g(W)-2}{2} $-connected.\end{proposition}

\begin{proof}By freely adjoining degeneracies we can make the semisimplicial space $X_\bullet(W)$ into a simplicial space $\tilde{X}_\bullet(W)$:
\[\tilde{X}_q(W) = \coprod_{[q] \twoheadrightarrow  [q']} X_{q'}(W)\]
This is Reedy cofibrant and has the property that the thin geometric realization of $\tilde{X}_\bullet(W)$ is isomorphic to the thick geometric realization of $X_\bullet(W)$. Consider the following two simplicial spaces, the second of which is constant:
\[[q] \mapsto \tilde{X}_q(W) \qquad [q] \mapsto \cM(W)\]

The augmentation $\epsilon$ can be used to produce a simplicial map $\tilde{\epsilon}$ between them which realizes to the map $||\epsilon||$. Since $ [q] \mapsto \pi_0(\cM(W))$ is a discrete simplicial set, by Proposition \ref{prop.rezk} the homotopy fiber of the map
\[|\tilde{\epsilon}|:  |\tilde{X}_\bullet(W)| \cong ||X_\bullet(W)||  \to \cM(W)\] over a $0$-simplex $Y$, is the same as the realization of the levelwise homotopy fiber over $Y$. 

We now compute the levelwise homotopy fiber. For fixed $q$, the map $\tilde{\epsilon}$ is given by
\[\coprod_{[q] \twoheadrightarrow  [q']} \mr{Emb}_\partial(W,\bR^\infty) \times_{\Top_\partial(W)} K_{q'}(W) \to \mr{Emb}_\partial(W,\bR^\infty) \times_{\Top_\partial(W)} *\]
and hence a Kan fibration by Lemma \ref{lem.kqkan} and part (ii) of Lemma \ref{lem.kanfibration}. Thus the levelwise homotopy fiber is the same as the levelwise fiber over $Y$, and the realization of the latter is given by $|\tilde{\epsilon}|^{-1}(Y)$. This is isomorphic to $||K_\bullet(Y)||$, which is $ \frac{g(W)-4}{2} $-connected by Proposition \ref{prop.kconn} as $Y$ is homeomorphic to $W$ rel boundary.\end{proof}

We next identify the $E^1$-page. To do so, we need transitivity and cancellation lemma's that are almost identical to Corollaries 4.4 and 4.5 of \cite{sorenoscarstability}. It uses $K^\delta_\circ(W)$ as in Definition \ref{def.kdelta}, essentially the $(0,0)$-simplices of $K_\bullet(W)$.

\begin{lemma}\label{lem.transitivity} Suppose that $g(W) \geq 4$, then for given any two vertices $(t_0,\varphi_0)$, $(t_1,\varphi)$ of $K^\delta_\circ(W)$, there exists a homeomorphism of $W$ sending $\varphi_0$ to $\varphi_1$, which is isotopic to the identity on $\partial W$.\end{lemma}

\begin{proof}One first does the case that the images of $\varphi_0$ and $\varphi_1$ are disjoint. In that case a regular neigborhood of $\partial W \cup \mr{im}(\varphi_0) \cup \mr{im}(\varphi_1)$ is abstractly homeomorphic to $W_{2,2} = (S ^{2n-1} \times [0,1])\# 2(S^n \times S^n)$ with two standard copies of $H$ in it. The same construction as in Corollary 4.4 of \cite{sorenoscarstability} gives a homeomorphism of $W_{2,2}$ which is the identity on the first boundary, isotopic to the identity on the second boundary and swaps the two copies of $H$. The case where only the cores of $\varphi_0$ and $\varphi_1$ are disjoint reduces to the previous case by first shrinking both copies of $H$ by an isotopy onto a smaller neighborhood of the image of $C$, until they are disjoint.

If $\varphi_0$ and $\varphi_1$ are not disjoint, it suffices to show we can find a path of vertices $K^\delta_\circ(W)$ connecting them. This follows from $g(W) \geq 4$, since the complex $K^\delta_\circ(W)$ is then path-connected by Proposition \ref{prop.kdelta}. 
\end{proof}

\begin{lemma}\label{lem.cancellation} Let $W$ and $W'$ be compact 1-connected topological manifold of dimension $2n \geq 6$ with given identifications $\partial W \cong S^{2n-1} \cong \partial W'$ and homeomorphism $W \# (S^n \times S^n) \to W' \# (S^n \times S^n)$ rel boundary. Suppose $g(W \# (S^n \times S^n)) \geq 4$, then there exists a homeomorphism $W \to W'$ rel boundary.
\end{lemma}

\begin{proof}After picking arcs connecting $(S^n \times S^n) \backslash \mr{int}(D^{2n})$ to the boundary as in Definition \ref{def.k}, we get a $0$-simplex $(t_0,\varphi_0)$ in $K^\delta_\circ(W \# (S^n \times S^n))$ and a $0$-simplex $(t_0',\varphi_0')$ in $K^\delta_\circ(W' \# (S^n \times S^n))$. Applying the inverse of the homeomorphism to $(t_0',\varphi_0')$ we get a second $0$-simplex $(t_1,\varphi_1)$ of $K^\delta_\circ(W \# (S^n \times S^n))$. Since $g(W \# (S^n \times S^n)) \geq 4$, Lemma \ref{lem.transitivity} says there is a homeomorphism of $W \# (S^n \times S^n)$ that is isotopic to the identity on the boundary taking $(t_0,\varphi_0)$ to $(t_1,\varphi_1)$. Hence their complements are homeomorphic rel boundary.\end{proof}

Now we are getting closer to proof of homological stability, we return to the context of Theorem \ref{thm.main}. That is, we fix a compact $1$-connected topological manifold $M$ of dimension $2n \geq 6$ with boundary $S^{2n-1}$ and define 
\[M_{\# g} = M \#_g (S^n \times S^n)\]

It is clear from this definition that $g(M_{\# g}) \geq g(M) + g$. To identify the $E^1$-page, we define $N_q = \bigcup_{i=0}^q (W^\epsilon_{2,1}+i\cdot e_1)$ and make a choice of $q$-simplex $\sigma_q \in K_q(N_q)$ of thick cores so that the inclusion $(S^{2n-1} \times (-\epsilon,0]) \cup \mr{im}(\sigma_q) \hookrightarrow N_q$ is an isotopy equivalence. See Figure \ref{fig.nsigma} for an example.

\begin{definition} Given $(N_q,\sigma_q)$, the map $\iota^\sigma_q: \cM(M_{\# g-q-1})  \to X_q(M_{\# g})$ is given by
\begin{align*}\mr{Emb}^\mr{lf}_\partial(M_{\# g-q-1},\bR^\infty) \times_{\Top_\partial(M_{\# g-q-1})}* &\to \mr{Emb}^\mr{lf}_\partial(M_{\# g},\bR^\infty) \times_{\Top_\partial(M_{\# g})} K^s_q(M_{\# g})\\
(X,*) &\mapsto ((X+(q+1) \cdot e_1) \cup N_q,\sigma_q)\end{align*}
\end{definition}

\begin{proposition}\label{prop.rese1} The map $\iota^\sigma_q : \cM(M_{\# g-q-1}) \to X_q(M_{\# g})$ is a weak equivalence for $0 \leq q \leq g-4$.\end{proposition}

\begin{proof}We fix $q$ and identify the simplicial set 
\[X_q(M_{\# g}) = \mr{Emb}^\mr{lf}_\partial(M_{\# g},\bR^\infty) \times_{\Top_\partial(M_{\# g})} K_q(M_{\# g})\]
Let $K^s_q(M_{\# g}) \subset K_q(M_{\# g})$ be the simplicial subset of those $(q+1)$-tuples of the form $((0,\varphi_0),\ldots,(q,\varphi_q))$, i.e. we restrict how the thick cores attach to the boundary in the germ of collar chart. Since the space of maps $\Delta^k \to \bR$ is contractible, it is easy to see that the inclusion $K^s_q(M_{\# g}) \hookrightarrow K_q(M_{\# g})$ is a weak equivalence. Thus by part (iii) of Lemma \ref{lem.kanfibration} we can replace $K_q(M_{\# g})$ with $K^s_q(M_{\# g})$ in this definition. This makes it easier to apply isotopy extension later.

\begin{figure}
 \centering{
 \resizebox{125mm}{!}{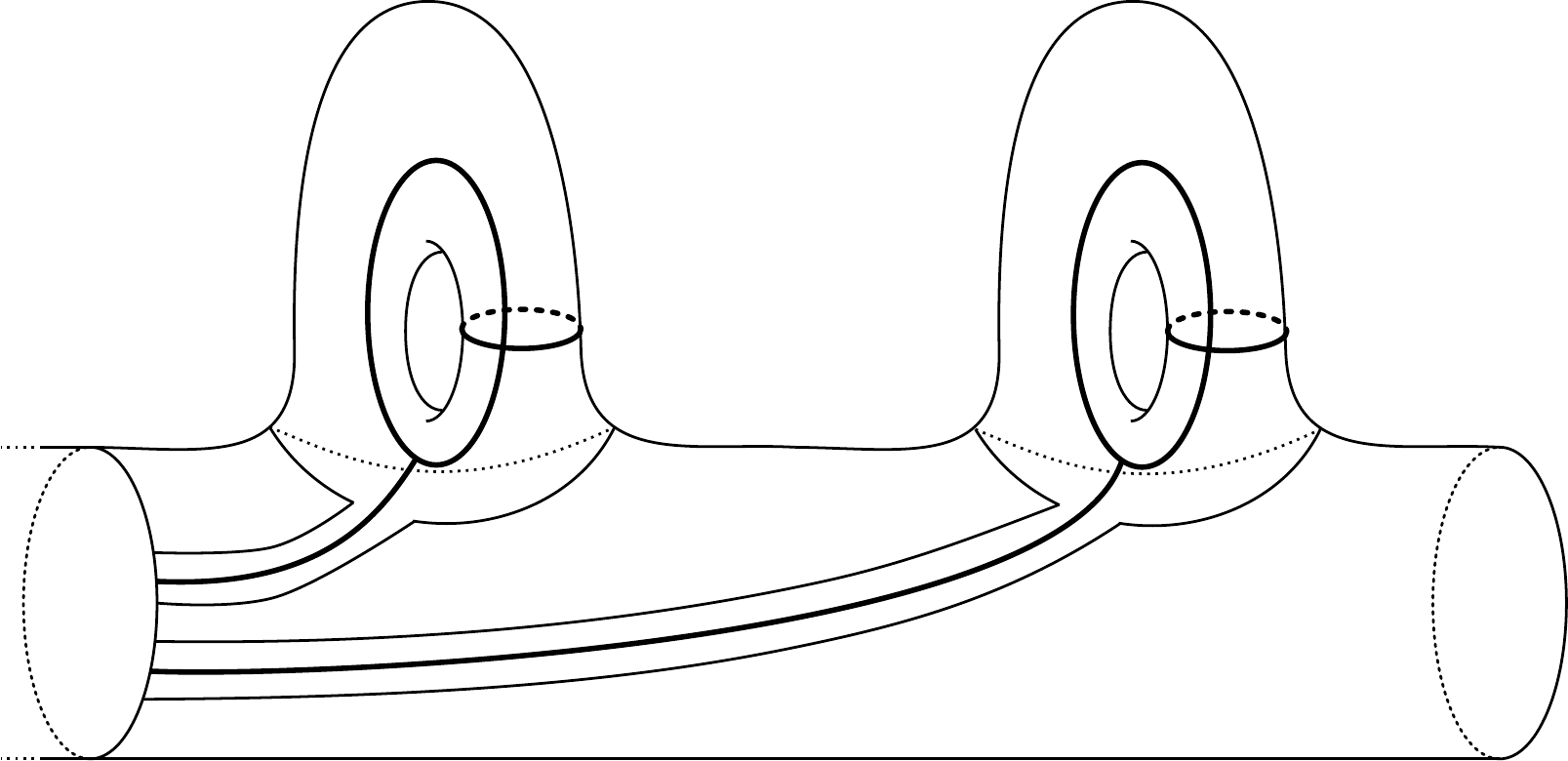}
 \caption{An example of $(N_1,\sigma_1)$ as in the proof of Proposition \ref{prop.rese1}.}
 \label{fig.nsigma}
}  \end{figure}

Consider the commutative diagram
\[\xymatrix{\Top_\partial(M_{\# g-q-1}) \ar[r]_-{\cup N_q} \ar@{->>}[d] & \Top_\partial(M_{\# g}) \ar@{->>}[d]  \\
\ast \ar[r]_-{\sigma_q} & K^s_q(M_{\# g})}\]
where the top map is obtained by extending homeomorphisms by the identity on $N_q$ and right map $\Top_\partial(M_{\# g}) \to  K^s_q(M_{\# g})$ is obtained by acting on $\sigma_q$. If all complements of $q$-simplices are homeomorphic to the complement $M_{\#g-q-1}$ of $\sigma_q$, then this map is surjective. By Lemma \ref{lem.cancellation}, this is in particular the case if $g \geq 4+q$. 

By isotopy extenson (\IE) relative to the boundary, the map $\Top_\partial(M_{\# g}) \to K^s_q(M_{\# g})$ a Kan fibration. Its fiber over $\sigma_q$ consists of those homeomorphisms of $M_{\# g}$ that are the identity on both the boundary and the image of $\sigma_q$. By our choice of $\sigma_q$, the induced map from the fiber of the left map, i.e. $\Top_\partial(M_{\# g-q-1})$, to the fiber of the right map is weakly equivalent to the identity.

Taking the product with embeddings over homeomorphisms, we get a commutative diagram
\[\xymatrix{\mr{Emb}^\mr{lf}_\partial(M_{\# g-q-1},\bR^\infty) \underset{\Top_\partial(M_{\# g-q-1})}{\times} \Top_\partial(M_{\# g-q-1}) \ar[r] \ar@{->>}[d] & \mr{Emb}^\mr{lf}_\partial(M_{\# g},\bR^\infty) \underset{\Top_\partial(M_{\# g})}{\times} \Top_\partial(M_{\# g}) \ar@{->>}[d]  \\
\mr{Emb}^\mr{lf}_\partial(M_{\# g-q-1},\bR^\infty)\underset{\Top_\partial(M_{\# g-q-1})}{\times} \ast \ar[r]_-{\iota^\sigma_q} & \mr{Emb}^\mr{lf}_\partial(M_{\# g},\bR^\infty) \underset{\Top_\partial(M_{\# g})}{\times} K^s_q(M_{\# g})}\]
and by Lemma \ref{lem.kanfibration} the vertical maps are still fibrations. The fiber of the left map is still $\Top_\partial(M_{\# g-q-1})$ and the map on fibers still is weak equivalent to the identity. Furthermore, the total spaces are contractible and since the original vertical maps were surjective, the bases are path-connected. Thus $\iota^\sigma_q$ is a weak equivalence.\end{proof}

Now that we have identified the entries on the $E^1$-page, we continue with the $d^1$-differential. It is important to note that the map of the previous proposition involved a choice of $N_q$ and $\sigma_q$. The next proposition will require picking them more carefully.

\begin{proposition}
\label{prop.resd1}
We can pick $(N_q,\sigma_q)$ so that resolution $X_\bullet(M_{\# g})$ has the following properties:
\begin{enumerate}[(i)]
\item For all $q$ the following diagram commutes
\[\xymatrix{\cM(M_{\# g-q-1}) \ar[d]_t \ar[r] & X_q(M_{\# g}) \ar[d]^{d_q}  \\ 
\cM(M_{\# g-q}) \ar[r] & X_{q-1}(M_{\# g})}\]
\item The map $\epsilon: X_0(M_{\# g}) \simeq \cM(M_{\# g-1}) \to \cM(M_{\# g})$ is homotopic to $t$.
\item For all $q$ and $0 \leq i \leq q-1$, $d_i$ is homotopic to $d_q$.
\end{enumerate}
\end{proposition}

\begin{proof}As in pages 21--24 of \cite{grwstab1}, for parts (i) and (ii), one can choose $\sigma_q$ so that $d_q \circ \iota^\sigma_q = \iota^{d_q(\sigma)}_{q-1} \circ t$ and $\epsilon \circ \iota_0^\sigma = t$.

For part (iii), we let $W_{p+1,2} := (S^{2n-1} \times [0,1])\# (p+1)(S^n \times S^n)$. There is an augmented semisimplicial simplicial set $Y_\bullet(p)$ given by $\mr{Emb}^\mr{lf}_\partial(W_{p+1,2},[0,q+1] \times \bR^\infty) \times_{\Top_\partial(W_{p+1,2})} K_\bullet(W_{p+1,2})$ and a map
\[Y_q(q) \times \cM(M_{\# g-q-1}) \to  \cM(M_{\# g})\]
such that restricting to $(N_q,\sigma_q) \in Y_q(q)$ we get the map above. Any other $(N'_q,\sigma'_q)$ in the same path component of $Y_q(q)$ induces the same map. We can assume $N'_q = N_q$, since all embeddings of $W_{q+1,2}$ in $\bR^\infty$ are isotopic by Lemma \ref{lem.embcontract}. In Lemma 6.14 of \cite{grwstab1}, Galatius and Randal-Williams construct a smooth example of $(N_q,\sigma_q)$ such that $d_i(N_q,\sigma_q)$ and $d_{i+1}(N_q,\sigma_q)$ are in the same path component. If we use their example, we get $d_i \sim d_q$.
\end{proof}

\subsubsection{Finishing the argument} The main theorem now follows by a standard spectral sequence argument:

\begin{proof}[Proof of Theorem \ref{thm.main}] Recall our goal is to prove that 
\[t_*: H_*\left(\cM(M_{\# g}) \right) \to H_*\left(B\Top_\partial(\cM(M_{\# g+1})\right)\]
is an isomorphism in the range $* \leq \frac{g-3}{2}$ and a surjection in the range $* \leq \frac{g-1}{2}$, using the geometric realization spectral sequence associated to the resolution $X_\bullet(M_{\# g})$ discussed in the previous propositions. We only give an outline of the argument, since it is well-known (see the proof of Theorem 6.3 of \cite{grwstab1} for details,  \cite{RW} for a nice exposition). 

The augmented geometric realization spectral sequence for $X_\bullet(M_{\# g})$ has $E^1$-page given by
\[E^1_{p,q} = \begin{cases} H_q(\cM(M_{\# g})) & \text{if $p = -1$} \\
H_q(X_p(M_{\# g})) & \text{if $p\geq 0$}\end{cases}\]
and $d^1$-differential $\sum_i (-1)^i (d_i)_*$. It converges to $0$ in the range $p+q \leq  \frac{g-2}{2} $. Using Propositions \ref{prop.rese1} and \ref{prop.resd1}, in the range of interest we can rewrite the $E^1$-page as 
\[E^1_{p,q} = \begin{cases} H_q(\cM(M_{\# g})) & \text{if $p = -1$} \\
H_q(\cM(M_{\# g-p-1})) & \text{if $p\geq 0$}\end{cases}\]
and the $d^1$-differential as alternatively $t_*$ and $0$. Thus the $E^2$-page has columns given by
\[E^2_{p,q} = \begin{cases} \mr{coker}[t_*:H_q(\cM(M_{\# g-p-2})) \to H_q(\cM(M_{\# g-p-1}))] & \text{if $p$ is odd} \\
\mr{ker}[t_*:H_q(\cM(M_{\# g-p-1})) \to H_q(\cM(M_{\# g-p}))] & \text{if $p$ is even}\end{cases}\] 

The proof is then by induction over $g$ of the statement that $t_ *$ is an isomorphism for $* \leq \frac{g-3}{2}$ and a surjection for $* \leq \frac{g-1}{2}$. By the induction hypothesis the $E^2$-page vanishes when $p \geq 1$ and $p+q \leq \frac{g-1}{2}$. This and the fact that the spectral sequence converge to $0$ in the range  $p+q \leq  \frac{g-2}{2} $, are the input to a spectral sequence argument that concludes in the vanishing of the columns $p=-1,0$ in a range, proving the inductive step.
\end{proof}

\subsection{The complex $K_\bullet$ and quadratic modules} The main input for the argument in the previous subsection was Proposition \ref{prop.kconn}, the connectivity of the complex of cores. This is where most of the technology of topological manifolds is used, and where the restriction to dimension $2n \geq 6$ comes in.

In this section $W$ will be a compact $1$-connected topological manifold of dimension $2n \geq 6$ with boundary $S^{2n-1}$. We will follow the Galatius-Randal-Williams argument in proving that $K_\bullet(W)$ is highly-connected via several auxiliary semisimplicial simplicial sets and simplicial complexes. We describe them informally below, and will give precise definitions later.

\begin{itemize}
\item $K^C_\bullet(W)$ is a version of $K_\bullet(W)$ where only the cores are required to be disjoint. 
\item  $K^\delta_\bullet(W)$ is the semisimplicial set of vertices of $K^C_\bullet(W)$.
\item $K^\delta_\circ(W)$ is the simplicial complex constructed out of $K^\delta_\bullet(W)$ by forgetting the ordering of the cores,
\item $K^\mr{alg}_\circ(I^\mr{Fr}_n(W),\lambda,\mu)$ is a simplicial complex of orthogonal hyperbolic summands of $n$-quadratic modules.
\end{itemize}

The steps are as follows:
\begin{enumerate}[(1)]
\item Galatius and Randal-Williams proved $K^\mr{alg}_\circ(I^\mr{Fr}_n(W),\lambda,\mu)$ is highly-connected.
\item By a lifting argument applied to a map $K^\delta_\circ(W) \to K^\mr{alg}_\circ(I^\mr{Fr}_n(W),\lambda,\mu)$
the complex $K^\delta_\circ(W)$ is shown to be highly-connected.
\item  The connectivity of $K^\delta_\bullet(W)$ is that of $K^\delta_\circ(W)$, since their realizations are homeomorphic. 
\item That $K^\delta_\bullet(W)$ is highly-connected will be used to show that $K^C_\bullet(W)$ is highly-connected using a microfibration argument.
\item Finally, the inclusion $K_\bullet(W) \hookrightarrow K^C_\bullet(W)$  will be shown to be a weak equivalence.
\end{enumerate}

\subsubsection{Quadratic modules} Our first goal will be to explain how to obtain an $n$-quadratic module from $W$. The essential ideas appear on page 152 of \cite{kirbysiebenmann}, but we will give additional details. Fix an embedding $b_W: D^n \times  \mr{int}(D^n) \hookrightarrow W$ and orientation on $W$ compatible with $b_W$, which exists since $W$ is $1$-connected.

\begin{definition}We define $I^\mr{Fr}_n(W)$ to be the set of pairs $(\phi,\gamma)$ of an element $\phi$ of $\pi_0(\mr{Imm}^\mr{lf}(S^n \times \mr{int}(D^n),W))$ and an equivalence class of regular homotopy $\gamma$ from $\phi|_{D^n \times  \mr{int}(D^n)}$ to $b_W$ (where we think of $D^n \times  \mr{int}(D^n) \subset S^n \times  \mr{int}(D^n)$ as the bottom hemisphere).\end{definition}

The frame bundle $\mr{Fr}(W)$ of a topological manifold is the principal $\mr{TOP}(2n)$-bundle associated to the $\bR^n$-bundle inside the tangent microbundle with structure group $\mr{TOP}(2n)$, which exists and is essentially unique by Kister's theorem (\KT).

\begin{lemma}$I^\mr{Fr}_n(W)$ is in bijection with $\pi_n(\mr{Fr}(W),b_W)$ and thus has the structure of an abelian group.\end{lemma}

\begin{proof}This is a consequence of immersion theory (\IT) in the guise of Theorem \ref{thm.it} and Proposition \ref{prop.sectident}, which depends on Kister's theorem (\KT).\end{proof}

Geometrically, one can identify this abelian group structure by connected sum at the base point using $\phi|_{D^n \times  \mr{int}(D^n)}$. We now define a pairing $\lambda$ and a map $\mu$ on $I^\mr{Fr}_n(W)$. This uses the homomorphism $I^\mr{Fr}_n(W) \to H_n(W)$ sending an immersed sphere to the image of its fundamental class.

\begin{definition}Define a pairing $\lambda: I^\mr{Fr}_n(W) \otimes I^\mr{Fr}_n(W) \to \bZ$ as the composite of $I^\mr{Fr}_n(W) \to H_n(W)$ and the intersection pairing.\end{definition}

\begin{lemma}\label{lem.lmwelldefined} Every framed immersion of a sphere is regularly homotopic to one with self-transverse core $S^n \times 0$. The function \[\mu: I^\mr{Fr}_n(W) \to \begin{cases}\bZ & \text{if $n$ is even} \\
\bZ/2\bZ & \text{if $n$ is odd} \end{cases}\]
which counts (signed) self-intersections of a representative with self-transverse core is well-defined.\end{lemma}

\begin{proof}To prove the first statement, use microbundle transversality (\MT) on the cores $S^n \times \{0\}$ of the immersions, which have normal microbundles because they extend to an immersion of $S^n \times  \mr{int}(D^n)$. 

To prove the second statement, we expand the outline in Appendix C.3 of \cite{kirbysiebenmann}. Take a regular homotopy $\Delta^1 \to \mr{Imm}^\mr{lf}(S^n \times \mr{int}(D^n),W)$ and consider it as a map $\Delta^1 \times S^n \times \mr{int}(D^n) \to \Delta^1 \times W$. By applying microbundle tranversality (\MT) to the map $\Delta^1 \times S^n \times \{0\} \to \Delta^1 \times W$ rel $\partial \Delta^1$ to make it self-transverse, we see that we can assume that the self-intersections of $\Delta^1 \times S^n \times \{0\}$ are given by a number of 1-dimensional manifolds with boundary in $\partial \Delta^1$, i.e. circles and intervals with endpoints in $\partial \Delta^1$. This means that self-intersections get born and die in pairs. For $n$ even we remark this happens with opposite signs.
\end{proof}

The following definition formalizes the properties of $\lambda$ and $\mu$.

\begin{definition}If $n$ is even, let $(\epsilon,\Lambda) = (1,\{0\})$. If $n$ is odd, let $(\epsilon,\Lambda) = (-1,2\bZ)$.  An \emph{$n$-quadratic module} is a triple $(I,\lambda,\mu)$ as follows:
\begin{enumerate}[(i)]
	\item  $\lambda: I \otimes I \to \bZ$ is a bilinear map such that $\lambda(x,y) = \epsilon \lambda(y,x)$,
	\item $\mu: I \to \bZ/\Lambda$ is a map that satisfies (a) $\mu(a \cdot x) \equiv a^2 \mu(x) \pmod{\Lambda}$ for $a \in \bZ$, (b) $\mu(x+y)-\mu(x)-\mu(y) \equiv \lambda(x,y) \pmod{\Lambda}$.
\end{enumerate}\end{definition}

We will now give two examples of $n$-quadratic modules.

\begin{example}Let $I = \bZ^2$, define $\lambda$ by $\lambda(e_1,e_2) = 1$, $\lambda(e_1,e_1) = \lambda(e_2,e_2) = 0$ and $\mu$ by $\mu(e_1) = \mu(e_2) = 0$. This gives $(I,\lambda,\mu)$ a structure of an $n$-quadratic module, called the \emph{hyperbolic module} and denoted by $H$.
\end{example}

\begin{proposition}\label{lem.ifrquad} $(I^\mr{Fr}_n(W),\lambda,\mu)$ is a $n$-quadratic module.
\end{proposition}

\begin{proof}We use a simply-connected version of the proof of Theorem 5.2 \cite{wall}, implicit in the remarks on page 152 of \cite{kirbysiebenmann}. Part (i) is obvious from the homological definition of $\lambda$. Part (ii)(b) follows from the fact that the self-intersections of a connected sum of two transversely intersecting spheres are given by the sum of those of individual spheres and the intersections of the spheres, if one takes the connected sum along a tube that avoids both spheres (which can be arranged by microbundle transversality \MT). Part (ii)(a) follows from the fact that $a \cdot x$ is represented by a connected sum of $a$ translates in $\mr{int}(D^n)$-direction of the same immersion: the self-intersections of this connected sum are the sum of the self-intersections and pairwise intersections. Since translation is a regular homotopy, by Lemma \ref{lem.lmwelldefined}, both the numbers of pairwise intersections and self-intersections is $\mu(x)$.\end{proof}

\begin{example}\label{examp.hquad} For $H$ as in Definition \ref{def.cores} and $W_{1,1} = (S^n \times S^n) \backslash \mr{int}(D^{2n})$, there is an inclusion $H \to (I^\mr{Fr}_n(W_{1,1}),\lambda,\mu)$ by sending $e_1$ to the regular homotopy class of a tubular neighborhood of the first sphere and $e_2$ to the regular homotopy class of a tubular neighborhood of the first sphere .\end{example}

We now define a complex of hyperbolic summands of an $n$-quadratic module, following \cite{grwstab1}.

\begin{definition}Let $Q = (I,\lambda,\mu)$ be an $n$-quadratic module. We define a simplicial complex  $K^\mr{alg}_\circ(Q)$ as follows.
\begin{itemize}
\item The vertices injective maps $h_i: H \to Q$ of $n$-quadratic modules.
\item A $(k+1)$-tuple $(h_0,\ldots,h_k)$ forms a $k$-simplex if all $h_i(H)$ are orthogonal.
\end{itemize}\end{definition}

Theorem 3.2 of \cite{grwstab1} proves the following connectivity result for $K^\mr{alg}_\circ(Q)$. Recall that for a simplex $\sigma$ of a simplicial complex $X_\circ$, the star $\mr{Star}_{\sigma}(X_\circ)$ consists of all simplices $\tau$ that have $\sigma$ as a face and the link $\mr{Link}_\sigma(X_\circ)$ is the subcomplex of $\mr{Star}_{\sigma}(X_\circ)$ spanned by vertices not in $\sigma$.

\begin{definition}A simplicial complex $X_\circ$ is said to be \emph{weakly Cohen-Macauley of dimension $\geq n$} if $X_\circ$ is $(n-1)$-connected and for all $p$-simplices $\sigma$, we have that $\mr{Link}_\sigma(X_\circ)$ is $(n-p-2)$-connected.\end{definition}

\begin{theorem}[Galatius-Randal-Williams] \label{thm.connalg} Suppose there is an injective map $H^{\oplus g} \to Q$ of $n$-quadratic modules. Then the complex $K^\mr{alg}_\circ(Q)$ is at least $ \frac{g-4}{2}$-connected. In fact, it is weakly Cohen-Macauley of dimension $\geq  \frac{g-2}{2}$.
\end{theorem}

\subsubsection{Lifting to topology} We have now reached the step where we lift from algebra to topology. This is the only part where $2n \geq 6$ is important. We start by describing the simplicial complex of interest.

\begin{definition}\label{def.kdelta} Fix a germ of collar chart $[\delta,\eta]$. We define a simplicial complex $K^\delta_\circ(W)$ as follows.
\begin{itemize}
\item The vertices given by an equivalence class as $\epsilon \to \infty$ of a pair $(t,\varphi)$ of a $t \in \bR$ and a locally flat embedding 
\[\varphi: H \cup ((-\epsilon,0] \times D^{2n-1}) \hookrightarrow W  \cup ((-\epsilon,0] \times S^{2n-1})\]
which satisfies $\varphi(s,r) = (s,r+te_2)$ for $(s,r)$ in the coordinates $(-\delta,\delta) \times D^{2n-1}$ of a representative of the germ of collar chart.
\item A $(q+1)$-tuple $((t_0,\varphi_0),\ldots,(t_q,\varphi_q))$ forms a $q$-simplex if the images of $C$ under the $\varphi_i$ are disjoint.
\end{itemize} 

\end{definition}

Lemma \ref{lem.ifrquad} implies that there is a map of simplicial complexes 
\[\pi \colon K^\delta_\circ(W) \to K^\mr{alg}_\circ(I^\mr{Fr}_n(W),\lambda,\mu)\]

We will use this map and Theorem \ref{thm.connalg} to prove that $K^\delta_\circ(W)$ is highly-connected, analogous to the lifting argument in Section 5 of \cite{grwstab1}. The proof is a bit more complicated than in \cite{grwstab1} due to transversality not being an open condition for topological manifolds. 

\begin{proposition}\label{prop.kdelta} We have that $K^\delta_\circ(W)$ is $ \frac{g(W)-4}{2}$-connected. In fact, it is weakly Cohen-Macauley of dimension $\geq \frac{g(W)-2}{2}$.
\end{proposition}

An important tool for proving this is the following coloring lemma, Theorem 2.4 of \cite{grwstab1}.

\begin{theorem}[Galatius-Randal Williams] \label{thm.coloring} Let $X_\circ$ be a simplicial complex that is weakly Cohen-Macauley of dimension $\geq n$, $f \colon S^{n-1} \to ||X_\circ||$ be a map and $h \colon D^n \to ||X_\circ||$ be a nullhomotopy. If $f$ is simplicial with respect to a PL triangulation  $S^{n-1} \cong ||L_\circ||$, then this triangulation extends to a PL triangulation
$D^n \cong ||K_\circ||$ and $h$ is homotopic rel $S^{n-1}$ to a simplicial map $g \colon ||K_\circ|| \to ||X_\circ||$ such
that
\begin{enumerate}[(i)]
\item for each vertex $v \in K_\circ \backslash L_\circ$, the star of $v$ intersects $L_\circ$ in a single (possibly
empty) simplex, and
\item for each vertex $v \in K_\circ \backslash L_\circ$, $g(\mr{Link}_v(K_\circ)) \subset \mr{Link}_{g(v)}(X_\circ)$.
\end{enumerate}
In particular, $g$ is simplexwise injective if $f$ is.
\end{theorem}


We start with a result along the lines of Lemma 5.5 of \cite{grwstab1}, with two modifications: we assume transversality of cores, and we additionally prove a result for links of simplices. 

\begin{lemma}\label{lem.lifttransverse} Suppose we are given a (possibly empty) $p$-simplex $\tau$ of $K^\delta_\circ(W)$. If we are given a diagram of simplicial maps
\[\xymatrix{S^i \cong ||L_\circ|| \ar[r]^-f \ar[d] & ||\mr{Link}_{\tau}(K^\delta_\circ(W))|| \ar[d]^\pi \\
D^{i+1} \cong ||K_\circ|| \ar[r]_-g \ar@{.>}[ru] & ||\mr{Link}_{\pi(\tau)}(K^\mr{alg}_\circ(I^\mr{Fr}_n(W),\lambda,\mu))||}\]
such that
\begin{enumerate}[(i)]
\item $g$ satisfies the conditions of Theorem \ref{thm.coloring}, and
\item for all pairs of vertices $k_0,k_1 \in K_\circ \backslash L_\circ$, the cores of $f(k_0)$ and $f(k_1)$ are transverse,
\end{enumerate} then we can find a dotted lift.
\end{lemma}

\begin{proof}To do this, we put a total order $k_1 \prec k_2 \prec \ldots $ on the vertices in $K_\circ \backslash L_\circ$ and construct a lift of $k_j$ by induction over $j$. By (i), it suffices to find a lift of $g(k_{j+1})$ when we have defined our lift on $g(k_1),\ldots,g(k_j)$, since in the initial case $j=0$ there is nothing to prove. Given $k_j$, we call the cores in the image under the lift of the full subcomplex of $K_\circ$ spanned by $L_\circ$ and $k_1,\ldots,k_{j-1}$, the \emph{previous cores}. During our lift, we will take care that the lift of $k_j$ is transverse to all previous cores.

The vertex $g(k_{j+1})$ is an embedding $H \hookrightarrow (I^\mr{fr}_n(W_\tau),\lambda,\mu)$ orthogonal to the image of $\pi(\tau)$. In particular, we get two regular homotopy classes of immersions $S^n \times  \mr{int}(D^n) \to W$ from this. Pick representatives $e$ and $f$, with cores $S_e$ and $S_f$. Using microbundle transversality (\MT) we can make their cores $S_e$ and $S_f$ 
\begin{enumerate}
\item self-tranverse, 
\item transverse each other, 
\item transverse to the previous cores,
\item transverse to the cores of $\tau$.
\end{enumerate}
Here we use that the previous cores are disjoint from the cores of $\tau$ and already pairwise transverse (this uses (ii)), so that by a preliminary application of microbundle transversality (\MT) we can make $S_e$ and $S_f$ disjoint from these intersections. As a consequence we can assume that the previous cores do not intersect at all, and we are only make $S_e$ and $S_f$ transverse to a single manifold. Each application of microbundle transversality moves a core by an isotopy, which can be extended by the isotopy extension (\IE) to ambient isotopy which moves the entire $S^n \times  \mr{int}(D^n)$.

We will now apply the Whitney trick (\WT) many times, to arrange that 
\begin{enumerate}
\item both $S_e$ and $S_f$ are embedded, 
\item $S_e$ and $S_f$ intersect microbundle transversally in a single point,
\item $S_e$ and $S_f$ are disjoint from any cores in the image under the lift of the intersection of $\mr{Link}(k_{j+1})$ with the full subcomplex spanned by $L_\circ$ and $k_1,\ldots,k_j$,
\item $S_e$ and $S_f$ are disjoint from the cores of $\tau$.
\end{enumerate}

As in illustrative example, we describe how to cancel a pair of locally transverse self-intersections of $S_e$. By the condition on the self-intersection number we can always pair self-intersections in pairs with opposite sign. Since $S^n$ is path-connected, we can find two arcs connecting the two self-intersection points and since $n \geq 3$ by microbundle transversality (\MT) we can assume these are disjoint locally flat embedded arcs. Furthermore, since $S_e$ intersects $S_f$ and the previous cores transversely, these arcs can be made to avoid intersection points with $S_f$ and the previous cores. The two arcs form a Whitney circle $c$. Using Theorem \ref{thm.wt6}, we can assume it has a local model in $W_\tau$ avoiding $S_e$. Here we use that $W_\tau$ is simply-connected since we are removing a self-transverse immersed $n$-dimensional manifold from a $2n$-dimensional manifold and $n \geq 3$. By microbundle transversality (\MT) we can assume the Whitney disk avoids $S_f$ and the previous cores, and by shrinking the model neighborhood we can assume the entire local model avoids $S_f$ and the previous cores. Now we can cancel the intersection points by a move in the local model. The Whitney trick is implemented by a compactly-supported isotopy of one of the two parts of $S_e$ in a local model, so by extending to an ambient isotopy in this local model we can move the entire $S^n \times \mr{int}(D^n)$-neighborhood of $S_e$. 

Note that during this procedure $S_e$ stays transverse to $S_f$, the previous cores and the cores of $\tau$. Furthermore, if $S_e$ was already disjoint from $S_f$, a previous core or a core of $\tau$, it stays disjoint. Similar remarks imply that arranging (b), (c) and (d) does not disrupt the fact that we have arranged the previous conditions. Thus it suffices to remark that by definition of $K^\mr{alg}_\circ$ the conditions of intersection numbers and self-intersection numbers for the Whitney tricks for (b), (c) and (d) are satisfied. 

Now restrict $e$ and $f$ to a $S^n \times D^n \subset S^n \times  \mr{int}(D^n)$, which we will proceed to plumb together to an embedding of $(S^n \times S^n) \backslash  \mr{int}(D^{2n})$. By microbundle transversality, there is a neighborhood $U$ of the intersection point $p$ of $S_e$ and $S_f$ with the following properties:
\begin{itemize}
\item $U$ does not intersect the previous cores or the cores of $\tau$,
\item there is a homeomorphism of \[(U \cap W_{g,1},U \cap S_e,U \cap S_f) \cong (\bR^{2n},\bR^n \times \{(0,\ldots,0)\},\{(0,\ldots,0)\} \times \bR^n)\]
\item $e: e^{-1}(U \cap S_e) \times D^n \to \bR^n \times \bR^n \cong \bR^{2n}$ is given by $(x,y) \mapsto (e(x),y)$,
\item $f: f^{-1}(\{p\}) \times D^n \to \bR^n \times \bR^n \cong \bR^{2n}$ is given by $(0,y) \mapsto (-y,0)$.
\end{itemize} 

We shall work in these coordinates. That is, we have $\bR^n \times D^n \subset \bR^{2n}$ and an embedding $f: \bR^n \times D^n$ such that $f$ restriction to $\bR^n \times \{0\}$ coincides with $(x,y) \mapsto (-y,x)$. To perform the plumbing, we need to show we can find a compactly supported isotopy of $f$ to an embedding which is given by the map $(x,y) \mapsto (-y,x)$ near the origin. By scaling $f$ in the fiber direction, we can assume the image of $\frac{1}{2}D^{2n}$ is contained in $D^{2n} \subset \bR^n \times D^n$ and contains $\frac{1}{4}D^{2n}$. Now apply Lemma \ref{lem.localid} to $g = r \circ f|_{\frac{1}{2}D^{2n}}$, with $r$ the homeomorphism $(x,y) \mapsto (-y,x)$, to get an embedding $\tilde{g}$ which equals the identity on $\frac{1}{4}D^{2n}$ and coincides with $g$ outside $\frac{1}{2}D^{2n}$. Then $f$ is isotopic to $r^{-1} \circ \tilde{g}$, which has the desired properties.

The result is an immersion of $W_{1,1} = (S^n \times S^n) \backslash \mr{int}(D^{2n})$ into $W$, which is an embedding near $S^2\vee S^2$ and such that $S^2\vee S^2$ is disjoint from all the previous cores and the cores of $\tau$. It is not yet connected to the boundary in a standard way. We thus pick an arc from $\partial W_{1,1}$ to a standard piece of arc $(s,t_0 e_2)$ in the coordinates $(-\delta,\delta) \times D^{2n-1}$ of the germ of collar chart, with $t_0$ strictly smaller than the minimum of the set of real numbers $t$ that appear in $g(k)$ for the vertices $k$ in the full subcomplex spanned by $L_\circ$ and $k_1,\ldots,k_j$, or in $\tau$. At the endpoints there are standard $D^{2n-1}$ extending the points. The subgroup of $S\Top(2n)$ of $\Top(2n)$ consisting of orientation-preserving homeomorphisms is path-connected, using the stable homeomorphism theorem (\SH). This allows us extend this path to a path with injective microbundle map. Immersion theory (\IT) gives us an immersed $[0,1] \times D^{2n-1}$ connecting these disks, and an application of microbundle transversality (\MT) makes $[0,1] \times  \{0\}$ embedded and disjoint from the cores, the previous cores and the cores of $\tau$. Shrinking the size of the disks rel $\partial([0,1]) \times D^{2n-1}$ gives the desired $[0,1] \times D^{2n-1}$ extending the immersion of $(S^n \times S^n) \backslash \mr{int}(D^{2n})$ to an immersion of $H$ with embedded core $C$, thus obtaining a vertex of $\mr{Link}_{\tau}(K^\delta_\circ(W))$. There is a minor issue of a possible spin flip (see the end of the proof of Lemma 5.4 of \cite{grwstab1}, but this can be fixed by modifying the framing of the embedding. This is the desired lift of $g(k_{j+1})$ and thus completes the induction step.\end{proof}

\begin{lemma}\label{lem.transverse}  Suppose we are given $W$ and a $p$-simplex $\tau$ of $K^\delta_\circ(W)$. The simplex is allowed to be empty and then $p=-1$. Further suppose that for $p'$-simplices $\tau'$ of $K^\delta_\circ(W)$ satisfying $p'>p$, we have shown that $\mr{Link}_{\tau'}(K^\delta_\circ(W))$ is $(\frac{g-2}{2}-p'-2)$-connected.

Then any map $f: S^i \to ||\mr{Link}_{\tau}(K^\delta_\circ(W))||$ with $i \leq \frac{g(W)-2}{2}-p-2$, is homotopic to a simplicial map $f': S^i \cong ||L_\circ|| \to ||\mr{Link}_{\tau}(K^\delta_\circ(W))||$ such that for all pairs of vertices $\ell_0,\ell_1 \in L_\circ$ the cores of $f'(\ell_0)$ and $f'(\ell_1)$ are transverse.
\end{lemma}

\begin{proof}The argument is analogous to a so-called badness argument. First of all, by simplicial approximation we can assume there is a PL triangulation $L_\circ$ of $S^i$ and $f$ is given by a simplicial map $L^\circ \to K^\delta_\circ(W)$. 

For the sake of induction we suppose we are given a collection $T \subset L_\circ$ of vertices such that all pairs $t_0,t_1 \in T$ we have that the cores of $f(t_0)$ and $f(t_1)$ are transverse. The initial case has $T = \varnothing$, we finish when $T$ contains all vertices. We call a simplex $\sigma$ of $L_\circ$ \emph{bad} if for each vertex $\ell$ of $\sigma$, the core of $f(\ell)$ is \emph{not} transverse to the cores of $f(t)$ for $t \in T$. We show how to decrease the number of bad simplices of maximal dimension by 1, changing $T$, $L_\circ$, and $f$ in the process. Eventually this process eliminates all bad simplices.

Suppose that $\sigma$ is a bad simplex of maximal dimension $0 \leq j \leq i$. Since $\sigma$ is maximal, $\mr{Link}_{\sigma}(L_\circ)$ maps to the subcomplex of $\mr{Link}_{f(\sigma) \cup \tau}(K^\delta_\circ(W))$ with of thick cores with core transverse to the core of each $f(t)$ for $t \in T$. Since $f(\sigma) \cup \tau$ is a simplex of dimension $p+1 \leq p' \leq p+j+1$ the hypothesis of the Lemma applies and $\mr{Link}_{f(\sigma) \cup \tau}(K^\delta_\circ(W))$ is at least $(\frac{g(W)-2}{2}-(p+j+1)-2)$-connected. 

Since $L_\circ$ is a PL triangulation of a PL manifold, we have $\mr{Link}_{\sigma}(L_\circ) \cong S^{i-j-1}$. 
The dimension of the sphere can be estimated by $i-j-1 \leq \frac{g(W)-2}{2}-p-2-j-1 = \frac{g(W)-2}{2}-(p+j+1)-2$.  Thus we can find a triangulated PL disk $D_\circ$ of dimension $i-j$ with $\partial D_\circ = \mr{Link}_{\sigma}(L_\circ)$ and a map $\bar{f}: D_\circ \to \mr{Link}_{f(\sigma) \cup \tau}(K^\delta_\circ(W))$ which extends $f|_{\partial D_\circ}$. Since all cores of $f(t)$ for $t \in T$ are transverse, by a small perturbation of all vertices $d \in \mr{int}(D_\circ)$ of the core of $\bar{f}(d)$, we can assume that these cores are all transverse to the cores of $f(t)$ for $t \in T$ and to each other. 


Since $\bar{f}$ has image in $\mr{Link}_{f(\sigma) \cup \tau}(K^\delta_\circ(W))$, we can combine $f$ and $\bar{f}$ to a map $f|_{\sigma} \ast \bar{f}: \sigma \ast D_\circ \to \mr{Link}_{\tau}(K^\delta_\circ(W))$. The complex is $\sigma \ast D_\circ$ is a $(i+1)$-dimensional PL disk and its boundary is given by glueing the PL disks $\sigma \ast \partial D_\circ = \mr{Star}_{L_\circ}(\sigma)$ and $\partial \sigma \ast D_\circ$ along the PL sphere $\partial \sigma \ast \partial D_\circ \cong \partial \mr{Star}_{L_\circ}(\sigma)$. Thus $f_\sigma \ast \bar{f}$ induces a homotopy from $f$ to a new map
\[f': L'_\circ = (L_\circ \backslash \mr{int}(\mr{Star}_{L_\circ}(\sigma))) \cup_{\mr{Star}_{L_\circ}(\sigma)} (\partial \sigma \ast D_\circ) \to \mr{Link}_{\tau}(K^\delta_\circ(W))\]
If we let $T'$ be the union of $T$ and the vertices in $\mr{int}(D_\circ)$, then this new map has one fewer bad simplices of dimension $j$ and no new bad simplices of dimension $>j$. To see this, note we removed $\sigma$, and the only new simplices are those involving a vertex in $\mr{int}(D_\circ)$ and these can not be bad since those vertices are in $T'$. Renaming $T'$, $L'_\circ$ and $f'$ to $T$, $L_\circ$, and $f$, we repeat this procedure to eventually make $T$ equal to the set of all vertices.\end{proof}


After these two lemma's we can prove Proposition \ref{prop.kdelta}.

\begin{proof}[Proof of Proposition \ref{prop.kdelta}] Recall that we need to prove that $K^\delta_\circ(W)$ is weakly Cohen-Macauley of dimension $\geq \frac{g(W)-2}{2}$. This means that for all $p$-simplices $\tau$, including the empty $-1$-simplex, we need to check that $\mr{Link}_{\tau}(K^\delta_\circ(W))$ is $(\frac{g(W)-2}{2}-p-2)$-connected. The proof is by downstairs induction over $p$. In the initial cases $p \geq \frac{g(W)-2}{2}$ there is nothing to prove. For the induction step we assume we have proven the statement for all $p' > p$, and we will prove it for $p$. After completing the step $p=-1$ we are done.

If we are given a map $f: S^i  \to ||\mr{Link}_{K^\delta_\circ(W)}(\tau)||$ for $i \leq  (\frac{g(W)-2}{2}-p-2)$, we need to extend it for the disk. By simplicial approximation we can assume $f$ is simplicial with respect to a triangulation $L_\circ$ of $S^i$. By the induction hypothesis, Lemma \ref{lem.transverse} applies and we can assume that for all vertices $\ell_0,\ell_1$ of $L_\circ$ the cores of $f(\ell_0)$ and $f(\ell_1)$ are transverse.

Using the map $H \hookrightarrow (I^\mr{Fr}_n(W_{1,1}),\lambda,\mu)$ of Example \ref{examp.hquad} and the $g(W)$ disjoint embeddings of $W_{1,1}$ into $W$ from Definition \ref{def.gw}, we get $g(W)$ orthogonal copies of $H$ in $(I^\mr{Fr}_n(W),\lambda,\mu)$. Using Theorem \ref{thm.connalg} we conclude that the link of $\pi(\tau)$ in $K^\mr{alg}_\circ(I^\mr{Fr}_n(W),\lambda,\mu)$ is $ (\frac{g(W)-2}{2}-p-2)$-connected.  Thus we can find a map $g$ making the following diagram commute
\[\xymatrix{S^i  \ar[r]^-f \ar[d] & ||\mr{Link}_{\tau}(K^\delta_\circ(W))|| \ar[d]^\pi \\
D^{i+1} \ar[r]_-g \ar@{.>}[ru] & ||\mr{Link}_{\pi(\tau)}(K^\mr{alg}_\circ(I^\mr{Fr}_n(W),\lambda,\mu))||}\]
A link of a $p$-simplex in a weakly Cohen-Macauley complex of dimension $\geq n$ is itself weakly Cohen-Macauley of dimension $\geq n-p-1$ by Lemma 2.1 of \cite{grwstab1}, so by Theorem \ref{thm.connalg} we can apply Theorem \ref{thm.coloring} and assume that there is a triangulation $K_\circ$ of $D^{i+1}$ so that $g$ is simplicial and satisfies the properties (i) and (ii) of that Theorem. To prove the proposition, it suffices to prove we can find a dotted lift making both triangles commute. But we have arranged it so that Lemma \ref{lem.lifttransverse} applies.
\end{proof}

\subsubsection{From discrete to topologized cores} We have now established the connectivity of the simplicial complex $K^\delta_\circ(W)$. There is a semisimplicial set $K^\delta_\bullet(W)$ with homeomorphic geometric realization, with $p$-simplices given by $(p+1)$-tuples $((t_0,\varphi_0),\ldots,(t_p,\varphi_p))$ such that $t_0<\ldots<t_p$. We will deduce the high connectivity of $K_\bullet(W)$ from that of $K^\delta_\bullet(W)_p$ using the following two lemma's, Propositions 2.6 and 2.7 of \cite{grwstab1}.

\begin{lemma}\label{lem.geomrel} If $X_\bullet \to Y_\bullet$ is a map of semisimplicial spaces and $X_p \to Y_p$ is $(n-p)$-connected, then the map of spaces $||X_\bullet|| \to ||Y_\bullet||$ is $n$-connected.
\end{lemma}

\begin{definition}A map of spaces $f: X \to Y$ is a microfibration if in each diagram:
\[\xymatrix{D^i \ar[r] \ar[d] & X \ar[d]^f \\
D^i \times [0,1] \ar[r] & Y}\]
there exists an $\epsilon \in (0,1]$ such that there is a lift of $D^i \times [0,\epsilon]$.\end{definition}

\begin{lemma}\label{lem.microfib} If $f: X \to Y$ is a microfibration with $n$-connected fibers, then $f$ is $(n+1)$-connected.\end{lemma}

We will use the following intermediate complex.

\begin{definition}$K^C_\bullet(W)$ is the following semisimplicial simplicial set:
\begin{itemize}
\item The $0$-simplices $K^C_0(W)$ in semisimplicial direction are equal to $K_0(W)$ as in Definition \ref{def.k},
\item The $q$-simplices $K_q(W)$ in semisimplicial direction, have $k$-simplices given by $(q+1)$-tuples $((t_0,\varphi_0),\ldots,(t_q,\varphi_q))$ of $k$-simplices of $K_0(W)$ such that $t_0<\ldots<t_q$ and for all $t \in \Delta^k$ the images of $C$ under the $\varphi_i$ are disjoint.
\end{itemize}\end{definition}

\begin{proposition}\label{prop.kcconn} We have that $K^C_\bullet(W)$ is $ \frac{g(W)-4}{2}$-connected.\end{proposition}

\begin{proof}Consider the bisemisimplicial simplicial set $K^C_{\bullet,\bullet}(W)$ with $(q_1,q_2)$-simplices given by the simplicial set of $(q_1+q_2+2)$-simplices of $K^C_\circ(W)$ such that the first $q_1+1$ are obtained from $K^\delta_\circ(W)$ by degeneracies (i.e. constant over $\Delta^k)$. 

There is a map $\iota: K^\delta_\bullet(W) \to K^C_\bullet(W)$ given by inclusion, and two augmentations $\epsilon: K^C_{\bullet,\bullet}(W) \to K^C_\bullet(W)$ and $\delta: K^C_{\bullet,\bullet}(W) \to K^\delta_\bullet(W)$. We have that $||\epsilon|| \simeq ||\iota|| \circ ||\delta||$ by Lemma 5.8 of \cite{grwstab1}. If we show that $||\epsilon||$ is $\frac{g(W)-2}{2}$-connected, we are done since $||\epsilon||$ is a $\frac{g(W)-2}{2}$-connected map factoring over the $\frac{g(W)-4}{2}$-connected space $||K^\delta_\bullet(W)||$. 

To show $||\epsilon||$ is a highly-connected, using Lemma \ref{lem.geomrel} it suffices to show that the map of simplicial spaces
\[\epsilon_{q_2}: ||K^C_{\bullet,q_2}(W)|| \to K^C_{q_2}(W)\]
is $(\frac{g-2}{2}-q_2-1)$-connected. As used in the proof of Proposition \ref{prop.resconn}, we can freely adjoin degeneracies to $K^C_{\bullet,q_2}(W)$ to obtain a bisimplicial set $\tilde{K}^C_{\bullet,q_2}(W)$ given by
\[[q] \mapsto \coprod_{[q] \twoheadrightarrow [q']} K^C_{q',q_2}(W)\]
such that $||K^C_{\bullet,q_2}(W)|| \cong |\tilde{K}^C_{\bullet,q_2}(W)| \cong |\mr{diag}(\tilde{K}^C_{\bullet,q_2}(W))|$. The advantage of $\mr{diag}(\tilde{K}^C_{\bullet,q_2}(W))$ is that we can use simplicial approximation: to show $||K^C_{\bullet,q_2}(W)|| \to K^C_{q_2}(W)$
is $(\frac{g-2}{2}-q_2-1)$-connected it suffices to show that for each pair $(K,\partial K)$ of finite simplicial sets with $(|K|,|\partial K|) \cong (D^{i+1},S^{i})$ with $i \leq \frac{g-2}{2}-q_2-2$ and each diagram
\[\xymatrix{\partial K \ar[r]^-g \ar[d] & \mr{diag}(\tilde{K}^C_{\bullet,q_2}(W)) \ar[d] \\
K \ar[r]_-G & K^C_{q_2}(W)}\]
we can find a lift of $|G|$ to a map of spaces $\tilde{G}: |K| \to |||K^C_{\bullet,q_2}(W)|||$ such that $\tilde{G}|_{|\partial K|} = |g|$.

By Lemma \ref{lem.lfglue}, $G$ can be represented by locally flat immersion $\varphi: D^{i+1} \times (\sqcup_{q_2+1} H) \to D^{i+1} \times W$ over $D^{i+1}$ that is an embedding when restricted to $\sqcup_{q_2+1} C$. The map $|g|$ is given by a map $S^i \to ||K^\delta_\bullet(W)||$ whose cores for $t \in S^{i-1}$ are disjoint from $\varphi|_t(\sqcup_{q_2+1} C)$. Consider the subspace of the space $D^{i+1} \times ||K^\delta_\bullet(W)||$ given by elements $(t,x)$ such that $x$ has cores avoiding those of $\varphi|_t$. The fibers of the projection to $D^{i+1}$ are given by $||\mr{Link}_{\varphi|_t}(K^\delta_\bullet(W))||$. By Proposition \ref{prop.kdelta}, these fibers are $(\frac{g-2}{2}-q_2-2)$-connected. Furthermore, since the condition of avoiding cores is open, the projection is a microfibration. Now use Lemma \ref{lem.microfib}. 
\end{proof}

\begin{proposition}\label{prop.kconn} We have that $K_\bullet(W)$ is $ \frac{g(W)-4}{2}$-connected.\end{proposition}

\begin{proof}We prove that the inclusion $i: K_\bullet(W) \hookrightarrow K^C_\bullet(W)$ is a weak equivalence. By Lemma \ref{lem.geomrel} it suffices to prove this levelwise, i.e. fixing $q \geq 0$. We proved that $K_q(W)$ is Kan in Lemma \ref{lem.kqkan}, and a similar argment gives that $ K^C_q(W)$ is Kan. Thus it suffices to find in each diagram
\[\xymatrix{\partial \Delta^i \ar[r]^-g \ar[d] & K_q(W) \ar[d]^i \\
\Delta^i \ar[r]_-G \ar@{.>}[ru]^{\tilde{g}} &  K^C_q(W)}\]
a map $\tilde{G}: \Delta^i \times \Delta^1 \to K^C_q(W)$ such that $\tilde{G}|_{\Delta^i \times 0} = G$ and which on $\partial \Delta^i \times \Delta^1$ is given by $i \circ g \circ \pi$ with $\pi: \partial \Delta^i \times \Delta^1 \to \partial  \Delta^i$ the projection, and a lift $\tilde{g}:  \Delta^i \to K^C_q(W)$ such that $i \circ \tilde{g} = \tilde{G}|_{ \Delta^i \times 1}$. 

We do this by shrinking the neighborhood of the cores as follows. There is a locally flat family of embeddings $\phi_s: W \to W$ with $s \in [0,\infty)$ such that \begin{enumerate}[(i)]
\item $\phi_0 = \mr{id}$,
\item $\phi_s$ is the identity on $C$,
\item $\mr{im}(\phi_s) \subset \mr{im}(\phi_{s'})$ if $s>s'$,
\item $\bigcap_s \mr{im}(\phi_s) = C$,
\item  $\phi_s$ is given by scaling in the $D^{2n-1}$-component on $[-1,0] \times D^{2n-1}$.
\end{enumerate}
Then $\tilde{G}$ is obtained from $G$ by precomposing the embeddings of $W$ with $\phi_s$ for $s \in [0,R] \cong \Delta^1$ for $R$ sufficently large.
\end{proof}

\section{Theorems about topological manifolds that serve as input}\label{sec.input} In this section we explain the definitions and theorems about topological manifolds relevant to this paper. We have taken the liberty of mentioning a few additional related results which are not used. In general, the references are \cite{siebenmannicm,kirbysiebenmann,freedmanquinn}. The following convention is standard and used in many of the references, even though we will not explicitly use it except in Corollary \ref{cor.whitney}. 

\begin{convention}
All our manifolds are Hausdorff and paracompact, hence metrizable.
\end{convention}

\subsection{Locally flat embeddings}\label{subsec.basics}

We start with the notion of an embedding of topological manifolds. For this to be well-behaved, we will need to require at the very least the local existence of a normal bundle.

\begin{definition}\begin{enumerate}[(i)]
\item A \emph{locally flat embedding} $f: M \to N$ is a map that is a homeomorphism onto its image and has the property that for all $m \in M$ there exists a neighborhood $m \in V \subset M$ and a map $V \times \bR^{n-m} \to M$ which extends $f|_V$ and is a homeomorphism onto its image.
\item Let $B$ be a topological manifold, possibly with boundary. A \emph{locally flat embedding $f: M \to N$ over $B$} is a map $f$ fitting into a commutative diagram
\[\xymatrix{B \times M \ar[rr]^f \ar[rd] & & B \times N \ar[ld] \\
& B &}\] which is a homeomorphism onto its image and has the property that for all $(t,m) \in B \times M$ there exists open neighborhoods $U$ of $t$ and $V$ of $m$ such that there is a map $U \times V \times \bR^{n-m} \to B \times N$ over $B$ which extends $f|_{U \times V}$ and is a homeomorphism onto its image. 
\item The \emph{simplicial set of locally flat embeddings} $\mr{Emb}^\mr{lf}(M,N)$ has as $k$-simplices the locally flat embeddings $f: \Delta^k \times M \hookrightarrow \Delta^k \times N$ over $\Delta^k$.
\end{enumerate} \end{definition}

We will occasionally be interested in embeddings of codimension $0$ manifolds with boundary.
\begin{definition}Suppose $M$ and $N$ have the same dimension and $M$ has boundary. Then the \emph{simplicial set of locally flat embeddings} $\mr{Emb}^\mr{lf}(M,N)$ has as $k$-simplices the embeddings $\Delta^k \times M \to \Delta^k \times N$ over $\Delta^k$ such that the restriction to $\partial M$ is locally flat.
\end{definition}

Alternative definitions of embeddings not only exist, but can be useful. We explain them in the following remark.

\begin{remark}\label{rem.embvariants} There are several choices for spaces of embeddings in the settings of topological manifolds: 
\[\mr{Emb}^\mr{glf}(M,N) \subset \mr{Emb}^\mr{lf}(M,N) \subset \mr{Emb}^\mr{plf}(M,N) \subset \mr{Emb}^\mr{w}(M,N)\]
\begin{enumerate}[(i)]
\item The first, $\mr{Emb}^\mr{glf}(M,N)$, are the \emph{globally flat embeddings}. That is, the $k$-simplices are $f: \Delta^k \times M \to \Delta^k \times N$ over $\Delta^k$ such that for all $t \in  \Delta^k$ there exists an open neighborhood $U$ of $t$, and an open $n$-dimensional manifold $V$ containing $M$ as a locally flat submanifold, such that there is a map $U \times V \to \Delta^k \times N$ over $\Delta^k$ which extends $f|_{U \times M}$ and is a homeomorphism onto its image. 
\item The second, $\mr{Emb}^\mr{lf}(M,N)$, are the locally flat embeddings defined above.
\item The third, $\mr{Emb}^\mr{plf}(M,N)$, are the \emph{pointwise locally flat embeddings}. That is, the $k$-simplices are maps $f: \Delta^k \times M \to \Delta^k \times N$ over $\Delta^k$ that are a homeomorphism onto their image and such that for all $(m,t) \in \Delta^k \times M$ there exists an open neighborhood $V$ of $m$ such that there is an embedding $\{t\} \times V \times \bR^{n-m} \to \{t\} \times N$ extending the original embedding.
\item The fourth, $\mr{Emb}^\mr{w}(M,N)$, are the \emph{wild embeddings}. That is, the $k$-simplices are maps $f: \Delta^k \times M \to \Delta^k \times N$ over $\Delta^k$ that are a homeomorphism onto their image.
\end{enumerate}

These variants are weakly equivalent in codimension at least 3. More precisely, from parametrized isotopy extension (\IE) it follows that $\mr{Emb}^\mr{glf}(M,N) = \mr{Emb}^\mr{lf}(M,N)$. If $m-n \geq 3$ and $m \geq 5$, then the remaining inclusions are weak equivalences by the Appendix of \cite{lashofembeddings}, which deduces it from a result of {\v{C}}ernavski{\u\i} \cite{cernavskiiemb}. 
\end{remark}

We now prove several useful properties of locally flat embeddings. It will be important that $\mr{Emb}^\mr{lf}$ is a Kan complex. A proof of this requires a glueing lemma for families of locally flat embeddings. 


\begin{lemma}\label{lem.lfglue} Let $B$ and $B'$ be manifolds with boundary, and $f: M \to N$ and $f': M \to N$ be locally flat embeddings over $B$ and $B'$ respectively. Suppose that $D \hookrightarrow \partial B$ and $D \hookrightarrow \partial B'$ are locally flat codimension $0$ submanifolds, and $f|_D = f'|_D$. Then $f \cup_D f'$ is a locally flat embedding over $B \cup_D B'$.\end{lemma}

\begin{proof}Since locally flatness is a local condition, without loss of generality $(B,D)$ is $(\Delta^k \times [0,1],\Delta^k \times \{0\})$ and $(B',D)$ is $(\Delta^k \times [-1,0],\Delta^k \times \{0\})$. By two applications of the parametrized isotopy extension theorem  (\IE) there exists two maps \begin{align*}\phi \colon \Delta^k \times [0,1] &\to \mr{TOP}(N) \\
\phi \colon \Delta^k \times [-1,0] &\to \mr{TOP}(N)\end{align*}
that are the identity on $\Delta^k \times \{0\}$ and such that for $(d,s) \in \Delta^k \times [0,1]$ we have $f(d,s) = \phi(d,s)(f(d,0))$ and for $(d,s) \in \Delta^k \times [-1,0]$ we have $f'(d,s) = \phi'(d,s)(f(d,0))$.

Take an $m_0 \in M$, open neighborhoods $U$ of $d \in \Delta^k$ and $V$ of $m_0$ in $M$ and map $F: U \times V \times \bR^{n-m} \to \Delta^k \times N$ over $\Delta^k$ extending $f|_{U \times V}$. Then the map exhibiting locally flatness of $f \cup_{D} f'$ is given by
\begin{align*}\bar{F}: U \times [-1,1] \times V \times \bR^{n-m}  &\to \Delta^k \times [-2,1] \times N \\
(d,t,m,s)&\mapsto \begin{cases}(d,t,\phi(d,t)(F(d,m,s))) & \text{if $t \geq 0$} \\
(d,t,\phi'(d,t)(F(d,m,s))) & \text{if $t < 0$}\end{cases}\end{align*}

\end{proof}

\begin{lemma}\label{lem.embkan} We have that $\mr{Emb}^\mr{lf}(M,N)$ is a Kan complex.\end{lemma}

\begin{proof}We need to check filling of horns $\Lambda^n_i \hookrightarrow \Delta^n$. By Lemma \ref{lem.lfglue} we get a locally flat embedding over $\Lambda^n_i$ that we need to extend to $\Delta^n$. We do this by simply pulling back along a retraction $\Delta^n \to \Lambda^n_i$, which can easily be seen to be locally flat.\end{proof}

Next we discuss the behavior of locally flat embeddings under products and use this to prove the weak Whitney embedding theorem.

\begin{lemma}\label{lem.lfprod} If $f: M \to N$ is locally flat over $B$ and $f': M \to N'$ is any map over $B$, then $(f,f'): M \to N \times N'$ is locally flat over $B$ as well.\end{lemma}

\begin{proof}Fix $(t,m) \in B \times M$ and let $U$ be neighborhood of $t$, $V$ a neighborhood of $m$, $\phi: U \times V \times \bR^{n-m} \to B \times N$ a map which extends $f|_{U \times V}$ be neighborhoods and is a homeomorphism onto its image. By shrinking $U$ and $V$ we can assume that $f'|_{U \times V}$ has image in a chart $\bR^{n'}$. To exhibit locally flatness of $(f,f')$ at $(t,m)$, we take $U$, $V$ as before and the map 
\[\tilde{\phi}(u,v,s,r) = (u,\phi(u,v,s),f'(u,v)+r)\]
where we write $B \times \bR^{n+n'-m}  = B \times \bR^{n-m} \times \bR^{n'}$ and the addition ``$+$'' in the second coordinate is given by the chart $\bR^{n'}$.\end{proof}

\begin{corollary}\label{cor.whitney} Every compact topological manifold $M$ admits a locally flat embedding into $\bR^N$ for $N$ sufficiently large.\end{corollary}

\begin{proof}By compactness there is a finite collection of charts $(U_i,\psi_i: U_i \cong \bR^m)_{1 \leq i \leq k}$ covering $M$ and by paracompactness there is a partition of unity $f_i$ subordinate to $\{U_i\}$. Let $N = k(m+1)$ and let the embedding be given by
\[m \mapsto (f_1(m) \cdot \psi_1(m),\ldots,f_k(m) \cdot \psi_k(m),f_1(m),\ldots,f_k(m)))\]

This can be seen to be locally flat by applying Lemma \ref{lem.lfprod} locally: note that on $f_i^{-1}((0,1])$ the embedding $m \mapsto (f_i \psi_i,f_i)$ is locally flat: it can be regarded as an embedding $f_i^{-1}((0,1]) \to \bR^m \times (0,\infty)$ and after composing with the homeomorphism $\bR^m \times (0,\infty) \to \bR^m \times (0,\infty)$ given by $(x,t) \mapsto (x/t,t)$ it is a restriction of the inclusion $\bR^m \hookrightarrow \bR^m \times (0,\infty)$ given by $m \mapsto (m,f_i(m))$ and one can use Lemma \ref{lem.lfprod} again.
\end{proof}

\subsection{Parametrized isotopy extension} Isotopy extension is a fundamental theorem in manifold theory, which often allows one to reduce questions about families of manifolds to questions about single manifolds.

The following is an amalgam of \cite{leesimm} (which only proves it in Euclidean spaces), \cite{cernavskiicontr}, Corollary 1.4 of \cite{edwardskirby} (which only proves the 1-parameter case), and the results of page 19 of \cite{burghelealashof}.

\begin{theorem}[Lees, {\v{C}}ernavski{\u\i}, Edwards-Kirby, Burghelea-Lashof] \label{thm.ie} Suppose $M$ is a compact topological manifold, $N$ is a topological manifold, and we are given a locally flat embedding $f$ fitting into a commutative diagram
\[\xymatrix{\Delta^k \times M \ar[dr] \ar@{^(->}[rr]^f & & \Delta^k \times N \ar[dl] \\
& \Delta^k & }\]
then there exists a homeomorphism $\psi$ fitting into a commutative diagram
\[\xymatrix{\Delta^k \times N \ar[dr] \ar_\cong[rr]^\psi & & \Delta^k \times N \ar[dl] \\
& \Delta^k & }\]
such that if $\Delta^0 \subset \Delta^n$ denotes the vertex $\{0\}$, then $\psi|_{\Delta^0} = \mr{id}_N$ and $f = \psi \circ (f|_{\Delta^0}): \Delta^k \times M \hookrightarrow \Delta^k \times N$.
\end{theorem}

\begin{remark}There is also a relative version, if one is given an extension to a homeomorphism on an open neighborhood of a compact subset of a manifold. This works for manifolds with boundary as long as the compact subset contains the boundary.
\end{remark}

The following Corollary appears in \cite{burghelealashof}, directly following from unwinding the definitions.
\begin{corollary}\label{cor.restrkan} Fix a locally flat submanifold $M \subset N$, then the evaluation map $\Top(N) \to \mr{Emb}^\mr{lf}(M,N)$ is a Kan fibration. There is also a relative version.
\end{corollary}

Though we will not use it, it is also a corollary that $\mr{Emb}^\mr{lf}(M,N) \to \mr{Emb}^\mr{lf}(M',N)$ is a Kan fibration if $M' \subset M$ is a locally flat submanifold \cite{lashofembeddings}.

\subsection{Microbundles}  To explain immersion theory and transversality, we need to discuss microbundles. In the topological setting these take the role that vector bundles play in the smooth setting. The basic reference is \cite{milnormicrobundles}.

\begin{definition}\label{def.mb} An $r$-dimensional \emph{microbundle} $\xi = (X,i,p)$ on a manifold $M$ is a space $X$ and a pair of maps $M \overset{i}{\longrightarrow} X \overset{p}{\longrightarrow} M$ such that 
\begin{enumerate}[(i)]
	\item $p \circ i =\mr{id}_M$,
	\item for all $m \in M$ there exists an open neighborhood $U$ of $m$ and $V$of $i(m)$ and a homeomorphism $\phi: V \cong U \times \bR^r$ making the following diagram commute
	\[\xymatrix{& V \ar[dd]^\cong_\phi \ar[rd]^p &  \\
	U \ar[ur]^i \ar[dr]_{- \times 0} & & U\\
	& U \times \bR^r \ar[ru]_{\pi_1} &}\]
\end{enumerate}\end{definition}

\begin{example}The definition of a manifold guarantees that 
\[M \overset{\Delta}{\longrightarrow} M \times M \overset{\pi_1}{\longrightarrow} M\]
is a microbundle of the same dimension as $M$. This is called the \emph{tangent microbundle}. This name is justified by the fact that if $M$ has a smooth structure, then the underlying microbundle of $TM$ is isomorphic to $\tau M$ (we will define isomorphisms of microbundle later). \end{example}

\begin{definition}Suppose $\xi = (X,i,p)$, $\xi' = (X',i',p')$ are microbundles over $M$, of dimension $r \leq r'$ respectively. A \emph{microbundle map} $\xi \to \xi'$ is an equivalence class of maps $F: U \to U'$ with $U \subset X$ and $U' \subset X'$ open subsets containing $i(M)$ and $i'(M)$ respectively, such that 
\begin{enumerate}[(i)]
\item $F(U) \subset U'$,
\item $F \circ i = i'$,
\item $p' \circ F = p$,
\item for all $m \in M$ there exists an open neighborhood $U$ of $m$, $V$ of $i(m)$, $V'$ of $i(m)$, homeomorphisms $\phi: V \cong U \times \bR^r$, $\phi': V' \cong U \times \bR^{r'}$ as in property (ii) in Definition \ref{def.mb}, and a homeomorphism $\tilde{F}: U \times \bR^{r'} \to U \times \bR^{r'}$ such that the $\tilde{F}|_{U \times \bR^r} = \phi' \circ F \circ \phi^{-1}$.
\end{enumerate}
The equivalence relation is that one is allowed to shrink $U$ and $U'$.\end{definition}

An isomorphism of microbundles is an microbundle map with an inverse. Using Theorem \ref{thm.kt}, Kister proved every $r$-dimensional microbundle contains an $\bR^r$-bundle with structure group $\Top(r)$, unique up to homeomorphism preserving the fibers and the zero section \cite{kister}. In particular it is unique up to microbundle isomorphism.

\begin{definition}\label{def.mbmaps} Let $\xi = (X,i,p)$ be a $r$-dimensional microbundle over $M$ and $X' = (X',i',p')$ a $r'$-dimensional microbundle over $M'$. The \emph{simplicial set of microbundle maps} $\mr{Bun}(\xi,\xi')$ has $k$-simplices pairs of a map $f$ fitting into a commutative diagram
\[\xymatrix{\Delta^k \times M \ar[rr]^f \ar[rd] & & \Delta^k \times M' \ar[ld] \\
& \Delta^k & }\]
and a microbundle map $F: \pi^* X \to (\pi' \circ f')^* X'$ (over $\Delta^k \times M$), where $\pi: \Delta^k \times M \to M$ and $\pi': \Delta^k \times M' \to M'$ are the projections.\end{definition}

These spaces of bundle maps are essentially homotopy-theoretic objects, see Essay V of \cite{kirbysiebenmann}, because the restriction maps are fibrations. More precisely, they are weakly equivalent to section spaces:

\begin{proposition}\label{prop.sectident} Let $\xi = (X,i,p)$ be a $r$-dimensional microbundle over $M$ and $X' = (X',i',p')$ a $r'$-dimensional microbundle over $M'$. Using Kister's theorem (\KT), pick an $\bR^r$-bundle $Z$ in $\xi$ and an $\bR^{r'}$-bundle $Z'$ in $\xi$. Then the map $\mr{Bun}(\xi,\xi') \to \mr{Map}(M,M')$ has homotopy fiber over $f:M \to N$ given by the space of sections of the fiber bundle over $M$ with fiber at $m$ the locally flat embeddings $Z_m \cong \bR^m \hookrightarrow \bR^n \cong Z'_{f(m)}$.\end{proposition}

The space of locally flat embeddings $\bR^m \hookrightarrow \bR^n$ is the \emph{topological Stiefel manifold} and is weakly equivalent to $\Top(n)/\Top(n,m)$, where $\Top(n,m)$ is the subgroup of homeomorphisms of $\bR^n$ fixing $\bR^m \times \{0\}$ pointwise \cite{lashofembeddings}.

\subsection{Immersion theory} A corollary of isotopy extension is immersion theory, an $h$-principle for immersions. We first describe the correct notion of immersion.

\begin{definition}\begin{enumerate}[(i)]
\item A \emph{locally flat immersion} of $M$ into $N$ is a continuous map $f: M \to N$ such that for all $m \in M$ there exists a neighborhood $V$ of $m$ such that $f|_V$ is a locally flat embedding.
\item Let $B$ be a topological manifold, possibly with boundary. A \emph{locally flat immersion $f: M \to N$ over $B$} is a map $f$ fitting into a commutative diagram
\[\xymatrix{B \times M \ar[rr]^f \ar[rd] & & B \times N \ar[ld] \\
& B &}\] which has the property that for all $(t,m) \in B \times M$ there exists open neighborhoods $U$ of $t$ and $V$ of $m$ such that $f|_{U \times V}$ is a locally flat embedding.
\item The \emph{simplicial set $\mr{Imm}^\mr{lf}(M,N)$ of locally flat immersions} has as $k$-simplices the locally flat immersions $\Delta^k \times M \to \Delta^k \times N$ over $\Delta^k$.
\end{enumerate}\end{definition}

\begin{remark}Just like embeddings, immersions come in variants: \[\mr{Imm}^\mr{glf}(M,N) \subset \mr{Imm}^\mr{lf}(M,N) \subset \mr{Imm}^\mr{plf}(M,N) \subset \mr{Imm}^\mr{w}(M,N)\]
Since the results for embeddings have strongly relative versions, we have that $\mr{Imm}^\mr{glf}(M,N) = \mr{Imm}^\mr{lf}(M,N)$ by \cite{edwardskirby} and $ \mr{Imm}^\mr{lf}(M,N) \simeq \mr{Imm}^\mr{plf}(M,N) \simeq \mr{Imm}^\mr{w}(M,N)$ if $n \geq 5$ and $n-m \geq 3$ by \cite{lashofembeddings,cernavskiiemb}.\end{remark}

Analogous to the fact that smooth immersions have an injective differential, a topological immersion induces a microbundle map (which by definition are injective). In other words, there exists a map $d:\mr{Imm}^\mr{lf}(M,N) \to \mr{Bun}(TM,TN)$. The following is proven in \cite{leesimm} and \cite{burghelealashof}.

\begin{theorem}[Lees, Burghelea-Lashof] \label{thm.it} Let $M$ and $N$ be compact topological manifolds of dimension $m$ and $n$ respectively. Suppose either $m < n$, or $m=n$ and $M$ has no compact components. Then the map $d:\mr{Imm}^\mr{lf}(M,N) \to \mr{Bun}(TM,TN)$ is a homotopy equivalence. There is also a relative version.
\end{theorem}

\begin{remark}Some of the statements in the literature of this theorem assume the existence of a handle decomposition, which is a problem since not all topological 4-manifolds have a handle decomposition. This restriction is removed in Appendix V.A of \cite{kirbysiebenmann}. The statement in \cite{leesimm} fixes an ambient manifold $M'$ that contains $M$ of the same dimension as $N$. The results in \cite{burghelealashof} and \cite{kirbysiebenmann} show this is not necessary.\end{remark}

\subsection{Smoothing, triangulation and the homotopy groups of automorphism groups} \label{subsec.stto} We are mainly interesting in immersions of codimension zero manifolds, and the previous section explains the relevance of $\Top(k)$ to this problem. The following is used in the introduction and Subsection \ref{subsec.pl}.

Since the homotopy groups of of $O(k)$ are known in a range by Bott periodicity, understanding the homotopy groups of $\Top(k)$ in a range reduces to understanding the homotopy groups of $\Top(k)/O(k)$. The following combines \cite{rado}, \cite{moise}, \cite{smaledisk}, \cite{hatchersmale}, Section V.5 of \cite{kirbysiebenmann} and Theorem 8.7A of \cite{freedmanquinn}.

\begin{theorem}\label{thm.to} If $k \leq 3$, we have $\Top(k)/O(k) \simeq *$. For $k > 3$, the map $\pi_i(\Top(k)/O(k)) \to \pi_i(\Top/O)$ is an isomorphism for $i \leq k+1$ and a surjection for $i = k+2$. 

We have that
\[\pi_i(\Top/O) \cong \begin{cases}0 & \text{if $i = 0,1,4$} \\
\bZ/2\bZ & \text{if $i=3$} \\
\Theta_i & \text{if $i \geq 5$}\end{cases}\]
where $\Theta_i$ is the group of diffeomorphism classes of smooth homotopy spheres of dimension $i$ under connected sum.
\end{theorem}

Note that $\Theta_i$ is known to be finite \cite{kervairemilnor}, i.e. there exist only finitely many smooth homotopy spheres of fixed dimension $\geq 5$. The previous theorem is proved using smoothing theory and the product structure theorem. Let us state the version of smoothing theory that will be of use for us \cite{kirbysiebenmann,burghelealashof}. Let $\theta_O: BO(k) \to B\Top(k)$ be obtained by replacing the map induced by $O(k) \to \Top(k)$ by a fibration. There is a general result that says that there is a weak equivalence between a space of smooth structures on $M$ and the space of lifts of $TM: M \to B\Top(k)$ to $BO(k)$, see Essays V \& VI of \cite{kirbysiebenmann} and \cite{burghelealashof}. It is traditional to break this into two pieces:

\begin{theorem}[Burghelea-Lashof, Kirby-Siebenmann] \label{thm.st} Let $W$ be a topological manifold of dimension $k$, without connected components if $k=4$. Then the set of isotopy classes of smooth structures on $W$ is in bijection with equivalence classes of smooth reductions of the tangent microbundle, i.e. vertical homotopy classes of lifts 
\[\xymatrix{ & BO(k) \ar[d]^{\theta_O} \\
W \ar[r] \ar@{.>}[ru] \ar[r]_-{TW} & B\Top(k)}\]
For each smooth structure on $W$ there is a map
\[\mr{hofib}(B\Diff(W) \to B\Top(W)) \to \mr{Bun}(TW,(\theta_O)^* \gamma)\]
which is an injection on $\pi_0$ and an isomorphism on $\pi_i$ for $i>0$.
\end{theorem}

The relationship between smoothing theory and the homotopy groups of $\Top(k)/O(k)$ is as follows: obstruction theory says that one can find a lift if and only if certain obstruction classes $O_i \in H^{i+1}(W;\pi_i(\Top(k)/O(k)))$ vanish. Note that in the range $i+1 \leq k$, $\pi_i(\Top(k)/O(k)) = \pi_i(\Top/O))$.

There is a similar theory for PL structures on topological manifolds, called triangulation theory. It proceeds along the lines of the smooth theory, with the single modification that the homotopy groups of $\Top(k)/\Pl(k)$ are simpler in high dimensions.

\begin{theorem}\label{thm.plhomotopy} If $k \leq 3$, $\Top(k)/\Pl(k) \simeq *$. For $k=4$, $\Top(k)/\Pl(k) \simeq \Top(k)/O(k)$. For $k > 4$, $\Top(k)/\Pl(k) \simeq K(\bZ/2\bZ,3)$.
\end{theorem}

\begin{theorem}[Burghelea-Lashof, Kirby-Siebenmann] \label{thm.tr} Let $W$ be a topological manifold of dimension $k$, without connected components if $k=4$. Then set of isotopy classes of PL structures on $W$ is in bijection with equivalence classes of PL reductions of the tangent microbundle, i.e. vertical homotopy classes of lifts 
\[\xymatrix{ & B\Pl(k) \ar[d]^{\theta_\Pl} \\
W \ar[r] \ar@{.>}[ru] \ar[r]_-{TW} & B\Top(k)}\]
For each PL structure on $W$ there is a map
\[\mr{hofib}(B\Pl(W) \to B\Top(W)) \to \mr{Bun}(W,(\theta_\Pl)^* \gamma)\]
which is an injection on $\pi_0$ and an isomorphism on $\pi_i$ for $i>0$.\end{theorem}

\subsection{Microbundle transversality} Next we describe an adequate notion of transversality for our purposes, though not necessarily satisfactory in general. A good candidate in general is Marin's stabilized microbundle transversality \cite{marin}, which we discuss later.

The way one proves transversality theorems is by reducing them to a statement about Euclidean spaces using charts. This requires the existence of a tubular neighborhood on one of the two manifolds.

\begin{definition}A \emph{normal microbundle} for a submanifold $M$ of $N$ is a microbundle $\nu$ over $M$ with an embedding $\nu \to N$ onto an open neighborhood of $M$.\end{definition}

In the smooth setting there is always a unique normal bundle, but in the topological or PL setting they are additional data: even if the embedding is locally flat they may not exist, and if they exist they may not be unique \cite{rsnormal}. They do exist and are unique after stabilizing by taking a product $\bR^s$ for $s$ sufficiently large, which is the basis of Marin's theorem \cite{marin}.

\begin{definition}Suppose $M$ is a submanifold of $N$ with a normal microbundle $\nu$. A submanifold $M'$ of $N$ is said to be \emph{microbundle transverse} to $\nu$ if $M \cap M'$ is a locally flat submanifold of $M$ and there exists a normal microbundle $\nu'$ of $M \cap M'$ in $M'$ such that $\nu' \to \nu$ is an open embedding of fibers.\end{definition}

Remark that the dimension of $\nu$ is $n-m$, so that this is also the codimension of $M \cap M'$ in $M'$. Thus $M \cap M'$ is $(m+m'-n)$-dimensional, as expected. The following is the main result of \cite{quinntrans}, removing a dimension restriction in Theorem III.1.5 of \cite{kirbysiebenmann} (note that these references use ``submanifold'' for locally flat submanifold, c.f. page 13 of \cite{kirbysiebenmann}).

\begin{theorem}[Kirby-Siebenmann, Quinn] \label{thm.mt} Suppose $M$ and $M'$ are proper locally flat submanifolds of $N$, $M$ has a normal microbundle $\nu$, there are closed subsets $C \subset D \subset N$, and $M'$ is transverse
to $\nu$ near $C$. Then there is an isotopy of $M'$, realizable by an ambient isotopy supported in any given neighborhood of $(D\backslash C) \cap M \subset M'$, to a submanifold microbundle transverse to $\nu$ near $D$.\end{theorem}

There is also a transversality results for maps, which is equivalent to the previous transversality result for submanifolds, see e.g. Theorem III.1.1 of \cite{kirbysiebenmann} and Section 9.6 of \cite{freedmanquinn}. There are several notions in the topological category related to microbundle transversality that we would like to mention:

\begin{itemize}
\item \emph{Local  transversality} One could imagine defining transversality as locally near $p \in M \cap X$ 
	\[(N,M,M') \cong (\bR^n,\bR^m \times \{(0,\ldots,0)\},\{(0,\ldots,0,1)\} \times \bR^{m'})\]
	That is, it is modeled on the intersection of two affine planes in general position. If this is true, we say that $M$ and $M'$ are \emph{locally transverse}.
	\item \emph{Stabilized microbundle transversality.} Suppose $M$ and $M'$ are proper submanifolds of $N$, then that $M'$ is \emph{stably microbundle transverse} to $M$ in $N$ if $M'$ is locally transverse to $M$ in $N$ and if there exists an integer $s \geq 0$ such that there exists a normal microbundle $\nu$ to $M \times 0$ in $N \times \bR^s$ so that $M' \times \bR^s$ microbundle transverse to $\nu$ in $N \times \bR^s$.
	\item \emph{General position.} A map $f: M \to N$ is said to be in \emph{general position} if the self-intersection locus has the expected dimension. More precisely, we need that there are closed locally finite subcomplexes $K_M \subset M$ and $K_N \subset N$ such that
	\begin{enumerate}[(a)]
	\item $K_M$ is the singularity set of $f$ (the subset where $f$ is not injectve), 
	\item $\dim K_M \leq 2m-n$,
	\item  $g$ sends $K_M$ onto $K_N$ by a PL map,
	\item $g|_{K \backslash K_M}$ is a locally flat embedding,
	\item  for all $k \in K_M$ there exists locally flat simplices $\Delta^m \subset M$ and $\Delta^n \subset N$ with $k \in \mr{int}(\Delta^m)$ such that $g$ sends $\Delta^m$ linearly into $\Delta^n$ and the PL structures of $\Delta^m$ and $\Delta^n$ are compatible with those of $K_M$ and $K_N$.
	\end{enumerate}  
	Appendix III.C of \cite{kirbysiebenmann} describes more general notion of general position.
\end{itemize}
%

Let us now explain the know relationships between these notions of transversality.

\begin{lemma}Microbundle transversality implies local transversality. Local transversality implies microbundle tranversality when (i) $n \geq 5$, $m = m' = n-1$, or (ii) $n\leq 4$.\end{lemma}

\begin{proof}The first statement appears on page 91 of \cite{kirbysiebenmann}, but we will prove it. There is nothing to prove when $m+m'< n$. Suppose that $m+m'\geq n$, $M$ has normal bundle $\nu$ and $M'$ is microbundle transverse to $M$. For $p \in M \cap M'$, the normal microbundle $\nu$ for $M$ in $N$ gives a homeomorphism of a neighborhood of $p$ in $(N,M)$ to $(\bR^n,\bR^m \times \{(0,\ldots,0)\})$ such that $p$ corresponds to the origin and the fibers of $\nu \to M$ correspond to the fibers of the projection $\bR^n \to \bR^m$ on the first $m$ coordinates.

Microbundle transversality says in these coordinates, $M \cap M'$ is given by a locally flat submanifold $P \subset \bR^m \times \{(0,\ldots,0)\}$ passing through the origin and it has a normal bundle $\nu'$ in $P$ whose fibers coincide with those of the projection $\bR^n \to \bR^m$ on the last $m$ coordinates. Hence we may assume $M' = P \times \bR^{n-m}$. Since $P \subset \bR^m$ is locally flat, there is a homeomorphism of  a neighborhood of $p$ in $(\bR^m,P)$ with $(\bR^m, \{(0,\ldots,0,1)\} \times \bR^{m+m'-n})$. By taking the product with the identity on $\bR^{n-m}$ we get the desired homeomorphism 
\[(N,M,M') \cong (\bR^n,\bR^m \times \{(0,\ldots,0)\},\{(0,\ldots,0,1)\} \times \bR^{m'})\]

The second statement follows from existence and uniqueness of normal bundles. This is true in codimension $\geq 2$ when $n \geq 5$ by \cite{brownnormal,ksnormal}. It is also true in dimension $\leq 3$, because $\Diff = \Top$ in that case \cite{moisebook,hatchertorus,moise}, and in dimension 4 by Section 9.3 of \cite{freedmanquinn}. 
\end{proof}

\begin{lemma}\label{lem.marin} Microbundle transversality implies stabilized microbundle transversality. Theorem \ref{thm.mt} holds without the assumption that $M$ has a normal microbundle if we replace microbundle transversality with stabilized microbundle transversality.\end{lemma}

\begin{proof}The first claim is obvious by taking $s=0$. The second is due to Marin \cite{marin} (see also Theorem III.1.11 of \cite{kirbysiebenmann}), originally with a dimension restriction which can be removed by page 160 of \cite{freedmanquinn}).\end{proof}

Self-transversality results for a locally flat immersion can be deduced from microbundle transversality by a chart-by-chart induction. Alternatively, the cases $2m \leq n$ of self-transversality are implied by general position. Hence we state a useful general position result due to Dancis \cite{dancis}:
\begin{theorem}[Dancis] \label{thm.gp} If $f: M \to N$ is a proper map of topological manifolds, $m \leq n-3$ and $m \leq \frac{1}{3}(2n-1)$, then $f$ is properly homotopic by an arbitrary small perturbation to a map in general position. There is also a relative version for $f$ that are locally flat embeddings on a submanifold already.\end{theorem}

General position techniques imply that a version of Theorem \ref{thm.mt} is true without the assumption of normal microbundles as long as $5 \leq m,m' \leq n-3$ and $m+m'-n \leq 3$, see Theorem 1.2 of \cite{dancis}.

\subsection{Kister's theorem and the stable homeomorphism theorem} In this section we state Kister's theorem and the stable homeomorphism theorem, two important results on the structure of homeomorphisms of Euclidean space. We use these to deduce a lemma on modifying an embedding $\bR^n \to \bR^n$ near a point. The first part of the following result appears in \cite{kister}, the second part is a consequence of the proof. 

\begin{theorem}[Kister] \label{thm.kt} Any family $f_s: \Delta^i \to \mr{Emb}(\bR^n,\bR^n)$ of embeddings has an isotopy $(f_{s,t})_{t \in [0,1]}$ to a family $f_{s,1}$ of homeomorphisms. If $f_s(D^n) \subset \mr{int}(2D^{n})$ for all $s$, we can make the following additional assumptions: 
\begin{enumerate}[(i)]
\item The isotopy is equal to $f_s$ on the unit disk, i.e. for all $t \in [0,1]$ we have that $f_{s,t} = f_s$ on $D^n$.
\item The homeomorphisms map $2D^n$ into $2D^n$, i.e. $f_{s,1}(2D^n) \subset 2D^n$.
\end{enumerate}\end{theorem}

It is not clear that a homeomorphism fixing a point can be made the identity near that point. To prove this, we will use the notion of a stable homeomorphism. 

\begin{definition}A homeomorphism of a topological manifold $M$ is said to be \emph{stable} if it is a finite composition of homeomorphisms that are the identity on some open subset.\end{definition}

The following was proven first in dimension 2 \cite{rado}, followed by dimension 3 \cite{moise}, dimension $\geq 5$ \cite{kirbystable} and dimension 4 \cite{quinnends}.

\begin{theorem}[Rad\'o, Moise, Kirby, Quinn] \label{thm.sh} Every orientation-preserving homeomorphism $\bR^n \to \bR^n$ is stable.\end{theorem}

We now deduce a lemma we will use make embeddings coincide near a point.

\begin{lemma}\label{lem.shcon} If $f: \bR^n \to \bR^n$ is an orientation-preserving homeomorphism, then $f$ is isotopic to a homeomorphism that is the identity on $D^n$, through an isotopy with compact support.\end{lemma}

\begin{proof}By Theorem \ref{thm.sh} it suffices to prove this for a homeomorphism $\phi$ that is the identity on some open subset $U$. Take a closed disk $D \subset U$ and by linearly interpolating the radius and center move it to $D^n$. By isotopy extension this path of locally flat embeddings of disks can be extended to a compactly supported isotopy. Conjugating $\phi$ with this isotopy gives the desired isotopy of $\phi$.
\end{proof}

\begin{lemma}\label{lem.localid} If $f: \bR^n \to \bR^n$ is an orientation-preserving embedding with $f(D^n) \subset \mr{int}(2D^{n}) \subset f(2D^n)$. Then $f$ is isotopic to an embedding that is the identity on $D^{n}$, through an isotopy with support in $2D^{n}$.\end{lemma}

\begin{proof}Theorem \ref{thm.kt} gives an isotopy $(f_t)_{t \in [0,1]}$ from $f_0 = f$ to a homeomorphism $f_1$, which we can assume is equal to $f$ on $D^{n}$ and satisfies $f_1(2D^{n}) \subset 2D^{2n}$. Lemma \ref{lem.shcon} implies that $f_1$ admits an isotopy $(f_t)_{t \in [1,2]}$ with compact support in $\bR^n$ such that $f_2$ is the identity on $D^{n}$. By a radial compression, without loss of general the support of this family lies in $2D^n$. Now consider $(f_tf_1^{-1}f_0)_{t \in [1,2]}: \bR^{n} \to \bR^{n}$. This has support in $f_0^{-1}f_1(2D^n) \subset f_0^{-1}(2D^n) \subset 2D^n$, when $t=1$ is equal to $f_0$, and when $t=2$ is the identity on $D^n$.\end{proof}

\begin{remark}This lemma may sound elementary, but it is not. For example, it implies the stable homeomorphism theorem (\SH): note that for any homeomorphism $f$ the lemma produces a second homeomorphism $g$ that is the identity on $D^n$ and coincides with $f$ on $2D^n$. Then $f = g \circ (g^{-1} \circ f)$, both of which are the identity on an open subset.\end{remark}

\subsection{The Whitney trick} The Whitney trick gives conditions under which one can simplify intersections of submanifolds of complimentary dimension.

Suppose we are given two locally flat submanifolds $M$ and $M'$ of dimension $m$ and $m' = n-m$ of the $n$-dimensional manifold $N$, that intersect locally transversally. Then we can count their intersection points with signs. Suppose that $a,b \in M \cap M'$ have opposite sign and can be connected by locally flat arcs $\gamma$ on $M$ and $\gamma'$ on $M'$. These form a circle $\chi$ in $N$, which we call a Whitney circle. If $m,n-m \geq 3$, then removing $M \cup M'$ from $N$ does not affect $\pi_1$. Thus if $N$ was simply-connected, $\chi$ bounds a disk  $D$ in $N \backslash (M \cup M')$.

A local model is provided by 
\begin{align*}N &= \bR^{n} = \bR^{m} \times \bR^{m'} \\
M &=\bR^m \times \{0\} \subset \bR^{m} \times \bR^{m'} \\ M' &= \{0\} \times \{(x,y) \mid y = 1- x^2\} \times \bR^{m'-1} \subset \bR^{m-1} \times \bR^2 \times \bR^{m'-1} \\
\gamma & = \{0\} \times \{(t,0)\mid t \in [-1,1]\} \times \{0\} \subset \bR^{m-1} \times \bR^2 \times \bR^{m'-1} \\
\gamma' &= \{0\} \times \{(t,1-t^2) \mid t \in [-1,1]\} \times \{0\} \subset \bR^{m-1} \times \bR^2 \times \bR^{m'-1} \\
\chi &= \gamma \cup \gamma' \\
D &= \{0\} \times \{(t,s) \mid t \in [-1,1],s \in [0,1-t^2]\} \times \{0\} \subset \bR^{m-1} \times \bR^2 \times \bR^{m'-1}
\end{align*} 

\begin{figure}[t]
\centering{
\resizebox{135mm}{!}{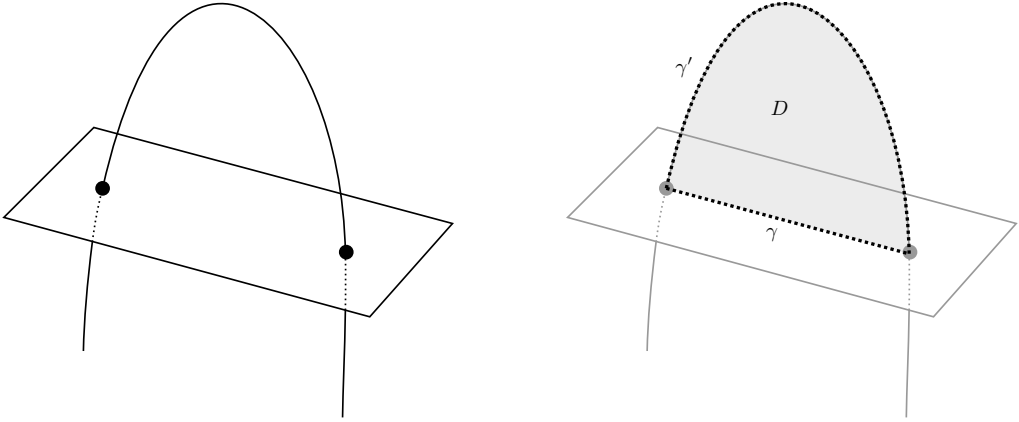}
\caption{The local model for the Whitney trick when $n=3$, $m=2$, $m'=1$.}
\label{fig.whitney}
}
\end{figure}

Note that in the local model we can cancel the two intersection points using a standard move which pushes $M$ upwards over $D$ as in Section 5 of \cite{milnorhcobord}. The Whitney trick for canceling intersection points is thus really about finding an embedded local model. The Whitney trick for topological manifolds it is stated in Section 7.3 of \cite{siebenmannicm} and proven in Section III.3 of \cite{kirbysiebenmann}.

\begin{theorem}[Kirby-Siebenmann] \label{thm.wt6} Suppose that $Y$ is 1-connected and of dimension $n \geq 6$, and $M$ and $M'$ are locally flat submanifolds of dimension $m$ and $m'=n-m$ respectively that intersect locally transversely. If $m,n-m \geq 3$, then any Whitney circle $c$ for $M$ and $X$ can be replaced by one that admits a local model. In particular, we can cancel the intersection points.
\end{theorem}

\begin{proof}The proof of the topological Whitney trick in Section III.3.4 of \cite{kirbysiebenmann} assumes the following: there exists an open neighborhood $U$ of $\chi = \gamma \cup \gamma'$ with a $\Diff$-structure such that $\gamma$ and $\gamma'$ are smooth in $U$ and $M$ and $M'$ are $\Diff$-transverse. We will show this can be arranged.

Using the local transversality, we get a chart $U_a$ in which the intersection point $a$ is $\Diff$-transverse and a chart $U_b$ in which $b$ is $\Diff$-transverse. It is easy to modify the arcs of the Whitney circle near these points so that they are smooth in $U_a$ and $U_b$. Then use the second part of Lemma \ref{lem.arcs} in $M$ or $M'$ to find charts $V$ and $V'$ agreeing with $U_a$ and $U_b$ near $a$ and $b$, and containing homotopic smooth embedded arcs $\gamma$ and $\gamma'$ respectively. By construction the smooth structures of $V$ and $V'$ are compatible near $a$ and $b$, so after possibly shrinking $V$ and $V'$ their union gives $U$.\end{proof}

\begin{lemma}\label{lem.arcs} Let $M$ be a path-connected topological manifold and $p,q \in M$. Then there exists a chart $V$ in $M$ containing an arc connecting $p$ and $q$, which is smooth in $V$.

Furthermore, if one starts with an embedded arc $\gamma$ and charts $U_0$, $U_1$ near the endpoints, we can arrange that the chart $V$ equals $U_0$ and $U_1$ near the endpoints and contains $\gamma$, and that the smooth embedded arc is homotopic to $\gamma$. If $\gamma$ was already smooth in $U_0$ and $U_1$ we can keep it fixed near its endpoints.\end{lemma}

\begin{proof}Take a path $\gamma$ connecting $p$ and $q$. Since in manifolds path-connectedness implies arc-connectedness we can assume it is embedded. By compactness there exists a finite collection of charts covering $\gamma$. We do induction of the number $k$ of such charts by eliminating one chart at a time. If $k=1$ we are done. Suppose $\gamma_0$ and $\gamma_1$ and are embedded arcs in $M$, $\gamma_0$ is contained in a chart $U_0$ and smooth in that chart, and $\gamma_1$ is contained in a chart $U_1$. Then we can find a chart $V$ containing $\gamma_0 \cup \gamma_1$. To do this, we see that from $U_1$ near $\gamma_1(0)$ we can obtain an embedding $\bR^n \to \bR^n$, which without loss of generality is orientation-preserving. Now apply Lemma \ref{lem.localid} (involving \KT\, and \SH), so that after shrinking $U_0$ and $U_1$ more we can glue the charts together. Now smooth $\gamma$ and make it embedded by smooth transversality. \end{proof}


\begin{remark}Alternatively one may use \emph{engulfing} to prove Lemma \ref{lem.arcs} and construct $U$ in the proof of Theorem \ref{thm.wt6}. This useful technique has many variations, an illustrative one appearing in \cite{newman} (see also \cite{rushing} for an exposition of various engulfing techniques). 

\end{remark}

\section{Extensions and applications}\label{sec.ext}

\subsection{Tangential structures and local coefficients}
Homological stability results for diffeomorphism groups admit several extensions. 

\begin{enumerate}[(i)]
\item The boundary of $M$ does not have to be a sphere, but can be any non-empty $(2n-1)$-dimensional manifold $P$.
\item Since $M$ might already contain some copies of $S^n \times S^n$, the relevant number in the range is not $g$ but $g(M_{\# g})$. 
\item Section 7 of \cite{grwstab1} extends the homological stability to diffeomorphisms of manifolds with $\theta$-structure, where $\theta: B \to BO(2n)$ is a spherical tangential structure. 
\item Section 8 of \cite{grwstab1} extends homological stability arguments for diffeomorphism groups to homology with abelian local coefficient systems. These are coefficient systems which have trivial monodromy along null-homologous loops.
\end{enumerate}

Following the proofs in \cite{grwstab1} while making modifications as in Section \ref{sec.proof}, one can prove the following strengthening of Theorem \ref{thm.main}. In our setting \emph{tangential structures} are maps $\theta: B \to B\Top(2n)$, a $\theta$-structure on $TW$ is a microbundle map $TW \to \theta^* \gamma$, and a tangential structure $\theta$ is \emph{spherical} if any $\theta$-structure over $D^{2n}$ extends to $S^{2n}$ (with $D^{2n}$ considered as a hemisphere).

Let $\gamma$ denote the universal microbundle over $B\Top(2n)$. and suppose we have lift of $TP \oplus \epsilon$ along $\theta$. For $W$ be a manifold with $P$-boundary, let $\mr{Bun}_\theta(W)$ be the subsimplicial set of $\mr{Bun}(TW,\theta^*\gamma)$ of microbundle maps that are equal to $l$ near the boundary. We define
\[B\Top^\theta_\partial(W) = \mr{Emb}^\mr{lf}_\partial(W,\bR^\infty) \times_{\Top_\partial(W)} \mr{Bun}_\theta(W)\]

\begin{atheorem}\label{thm.ext} Let $M$ a 1-connected topological manifold of dimension $2n \geq 6$ with non-empty boundary $P$ and $\theta: B \to B\Top(2n)$ be a spherical tangential structure with fixed lift $l$ of $TP \oplus \epsilon$.
\begin{enumerate}[(i)]
\item The stabilization map induces an isomorphism
\[H_*(B\Top^\theta_\partial(M_{\# g})) \to H_*(B\Top^\theta_\partial(M_{\# g+1}))\]
in the range $* \leq  \frac{g(M)+g-3}{2} $.
\item Let $\cA$ be an abelian coefficient system. The stabilization map induces an isomorphism
\[H_*(B\Top^\theta_\partial(M_{\# g});\cA) \to H_*(B\Top^\theta_\partial (M_{\# g+1});\cA)\]
in the range $* \leq  \frac{g(M)+g-4}{3} $.
\end{enumerate}
\end{atheorem}

\subsection{Application to PL homeomorphisms}\label{subsec.pl} One example of a tangential structure is the map $\theta_\Pl: B\Pl(2n) \to B\Top(2n)$ induced by the inclusion $\Pl(2n) \hookrightarrow \Top(2n)$ of simplicial groups. A PL structure on $W$ induces a microbundle map $TW \to (\theta_\Pl)^* \gamma$, i.e. gives a $\Pl(2n)$-reduction.

\begin{proposition}\label{prop.traingth}Let $W$ be a PL manifold of dimension $2n \geq 6$, then $B\Pl_\partial(W) \to B\Top^{\theta_\Pl}_\partial(W)$ is homotopy equivalent to the inclusion of a connected component.\end{proposition}

\begin{proof} There is a commutative diagram
\[\xymatrix{B\Pl_\partial(W) \ar[r] \ar[d] & B\Top^{\theta_\Pl}_\partial(W) \ar[d] \\
B\Top_\partial(W) \ar@{=}[r] & B\Top_\partial(W)}\] 
Triangulation theory as Theorem \ref{thm.tr} says that the homotopy fiber $F$ of $B\Pl_\partial(W) \to B\Top_\partial(W)$ maps to the homotopy fiber $F'$ of  $B\Top^{\theta_\Pl}_\partial(W) \to B\Top_\partial(W)$ by a map that is an injection on $\pi_0$ and an isomorphism on $\pi_i$ for $i > 0$. We claim that $B\Pl_\partial(W) \to B\Top^{\theta_\Pl}_\partial(W)$ is the inclusion of a connected component. The long exact sequences of homotopy groups start out as
\[\xymatrix@C=1em{\cdots \ar[r] & \pi_1(F) \ar[r] \ar@{=}[d] & \pi_1(B\Pl_\partial(W)) \ar[r] \ar[d] & \pi_1(B\Top_\partial(W)) \ar@{=}[d] \ar[r] & \pi_0(F) \ar[r] \ar@{^(->}[d] & 0 \ar[r] \ar@{^(->}[d] & 0 \ar@{=}[d] \\
\cdots \ar[r] & \pi_1(F') \ar[r] & \pi_1(B\Top^{\theta_\Pl}_\partial(W)) \ar[r] & \pi_1(B\Top_\partial(W))  \ar[r] & \pi_0(F') \ar[r] & \pi_0(B\Top^{\theta_\Pl}_\partial(W)) \ar[r] & 0 
}\]
and the five lemma then says the map $\pi_i(B\Pl_\partial(W)) \to \pi_i(B\Top^{\theta_\Pl}_\partial(W))$ for $i \geq 1$ is an isomorphism.
\end{proof}

Often the relationship is simpler:

\begin{remark}\label{rem.comppltop} If $n \geq 3$, by Theorems \ref{thm.plhomotopy} and \ref{thm.tr} the set of PL structures on a topological manifold $W$ up to isotopy is in non-natural bijection with $H^3(W,\bZ/2\bZ)$ (it may be helpful to note that the set of PL structures up to PL homeomorphism is a quotient of this, so that $\pi_0(B\Top^{\theta_\Pl}_\partial(W))$ stabilizes even if $n=3$). Thus if $H^3(W,\bZ/2\bZ) = 0$, the map in Proposition \ref{prop.traingth} is a weak equivalence. Furthermore, if $W$ is $3$-connected, then $\mr{Bun}_{\theta_\Pl}(W)$ is contractible by obstruction theory and hence $B\Pl_\partial(W) \to B\Top_\partial(W)$ is a weak equivalence.\end{remark}

To apply Theorem \ref{thm.ext} we check that PL structures are spherical.

\begin{lemma}\label{lem.triangsph} Every PL structure on $D^{2n}$ extends to $S^{2n}$.
\end{lemma}

\begin{proof}Cone off the PL structure on $\partial D^{2n}$ to extend to the upper hemisphere $D^{2n} \subset S^{2n-1}$.\end{proof}

This implies that Theorems \ref{thm.main} and \ref{thm.ext} extend to PL homeomorphisms: ``exhibiting homological stability'' is shorthand for saying the analogous results to Theorems \ref{thm.main} and \ref{thm.ext} hold.

\begin{corollary}\label{thm.pl}Let $W$ be a compact 1-connected PL manifold of dimension $2n \geq 6$ with non-empty boundary $P$. Then  $B\Pl_\partial(M_{\# g})$  exhibits homological stability.\end{corollary}

\begin{proof}Combine Proposition \ref{prop.traingth} and  Lemma \ref{lem.triangsph} with Theorems \ref{thm.main} and \ref{thm.ext}.
\end{proof}

\begin{remark}\label{rem.plrefs} This corollary can also be proven by repeating the proof given in Section \ref{sec.proof} and noting that all the ingredients are true in the PL category. Here are the relevant references:
\begin{description}
	\item[IE] Parametrized isotopy extension, \cite{hudson}.
	\item[IT] Immersion theory, \cite{haefligerimm}.
	\item[MT] Microbundle transversality, \cite{rourkesandersonblock2}.
	\item[WT] Whitney trick, Theorem 5.12 of \cite{rourkesanderson}.
	\item[KT] Kister's theorem, \cite{kuiperlashof1,kuiperlashof2}.
	\item[SH] Stable homeomorphism theorem, trivially true.
\end{description}\end{remark}

\subsection{Application to discrete homeomorphisms} Let $\Diff^\delta(W_{g,1};\partial)$ denote the \emph{discrete} group of diffeomorphisms that are the identity near the boundary. Nariman proved homological stability for $B\Diff^\delta(W_{g,1};\partial)$ when $2n =2$ or $2n \geq 6$ \cite{narimanhigh,nariman2}. His techniques also apply to $M_{\# g}$ for any compact 1-connected smooth manifold $M$ with non-empty boundary $P$. Using the techniques of Section \ref{sec.proof} or Remark \ref{rem.plrefs} one can extend this result to homeomorphisms or PL homeomorphisms. However, in the case of homeomorphisms we can use a remarkable theorem of McDuff to give a quick proof.

Let $\Top(M_{\# g};\partial)$ denote the simplicial subgroup of $\Top_\partial(M_{\# g})$ consisting of homeomorphisms that are the identity near the boundary and $\Top^\delta(M_{\# g};\partial)$ the \emph{discrete} group of $0$-simplices. Inclusion induces a map
\[B\Top^\delta(M_{\# g};\partial) \to B\Top(M_{\# g};\partial)\]
which was proven by McDuff to be a homology equivalence with all local coefficients coming from the target \cite{Mc3}. Furthermore, the inclusion $\Top(W_{g,1};\partial) \to \Top_\partial(M_{\# g})$ is a weak equivalence of simplicial groups. This implies that Theorems \ref{thm.main} and \ref{thm.ext} extend to discrete homeomorphisms.

\begin{corollary}\label{thm.disc}Let $M$ be a compact 1-connected topological manifold of dimension $2n \geq 6$ with non-empty boundary $P$. Then $B\Top^\delta(M_{\# g};\partial)$ exhibits homological stability.\end{corollary}

\bibliographystyle{amsalpha}
\bibliography{cell}

\end{document}